\documentclass[reqno]{amsart}
    \usepackage{float} 
    \usepackage{amsmath,amsthm}    
    \usepackage{latexsym}
    \usepackage{amsfonts}
    \usepackage{amssymb} 
    \usepackage{color} 
    \usepackage[shortlabels]{enumitem}
    \theoremstyle{plain}
    \newtheorem{theorem}{Theorem}
    \newtheorem{corollary}[theorem]{Corollary}
    \newtheorem{lemma}[theorem]{Lemma}
    \newtheorem{proposition}[theorem]{Proposition}
    \theoremstyle{definition}

    \newtheorem{remark}[theorem]{Remark}
    \newtheorem*{remark*}{Remark}

    \newcommand{\rd}{\mathbb R^d}
    \newcommand{\R}{\mathbb R}
    \newcommand{\Z}{\mathbb Z}
    \newcommand{\pr}{\mathbf P}
    \newcommand{\e}{\mathbf E}

\begin{document}

 \title[Random walks in half space]
 {Green function for an asymptotically 
 stable random walk in a half space}
 \thanks{
        D. Denisov was supported by a Leverhulme Trust Research Project Grant  RPG-2021-105. 
        V. Wachtel was supported by DFG
    }
 \author[Denisov]{Denis Denisov}
 \address{Department of Mathematics, University of Manchester, Oxford Road, Manchester M13 9PL, UK}
 \email{denis.denisov@manchester.ac.uk}
 
\author[Wachtel]{Vitali Wachtel}
\address{Faculty of Mathematics, Bielefeld University, Germany}
\email{wachtel@math.uni-bielefeld.de}

    \begin{abstract}
        We consider an asymptotically stable  multidimensional random walk $S(n)=(S_1(n),\ldots, S_d(n) )$.
        Let  $\tau_x:=\min\{n>0: x_{1}+S_1(n)\le 0\}$  be the first time 
        the random walk $S(n)$ leaves the upper half-space. 
        We obtain  the asymptotics of $p_n(x,y):= \pr(x+S(n) \in y+\Delta, \tau_x>n)$ 
        as $n$ tends to infinity, where $\Delta$ is a fixed cube. 
        From that we obtain the local asymptotics for the Green function  $G(x,y):=\sum_n p_n(x,y)$, as 
        $|y|$ and/or $|x|$ tend to infinity.  
    \end{abstract}
    
    
    \keywords{Random walk, exit time, harmonic function, conditioned process}
    \subjclass{Primary 60G50; Secondary 60G40,60J45, 60F17}
    \maketitle

    \section{Introduction, main results and discussion}
    
\subsection{Notation and assumptions.}
 Consider a random walk $\{S(n),n\geq 0 \}$ on $\rd$, $d\geq1$, 
    where 
    \[
    S(n) =X(1)+\cdots+X(n), \quad n\ge 1
    \] 
    and $\{X(n), n\geq1\}$ are independent copies of a random vector   
    $X=(X_1,\ldots,X_d)$. 
    For  
    $x=(x_1, \ldots, x_d)$ 
    in the (non-negative) half space, that is for $x_1\ge 0$,
    let 
    \[
        \tau_x:=\min\{\, n\ge 1 : x+S(n)\notin \mathbb H^+ \,\}
        =\min\{\, n\ge 1 : x_1+S_1(n)\le 0 \,\}
    \]
be the first time the random walk exits the (positive) half space
\[
\mathbb H^+=\{(x_1,\ldots, x_d): x_1> 0\}.
\]
When $x=0$ we will omit the subscript and write 
\[
    \tau:=\tau_0=\min\{\, n\ge 1 : S(n)\notin \mathbb H^+  \,\}
    =\min\{\, n\ge 1 : S_1(n)\le 0 \,\}. 
\]
In  this paper we study the asymptotic, as $n\to\infty$, behaviour of the probability
\begin{equation}\label{eq:green.n}
    p_n(x,y):=\pr(x+S(n)\in y+\Delta,\tau_x>n) 
\end{equation}
and the Green function 
\[
    G(x, y) := \sum_{n=0}^\infty p_n(x,y).
\]
Here and throughout we denote $\Delta=[0,1)^d$ and 
    for $y=(y_1,\ldots,y_d)$, 
    \[
        y+\Delta = [y_1,y_1+1)\times [y_2,y_2+1)\times\cdots \times[y_d,y_d+1).
    \]

In this paper we will mostly concentrate on the case when the random walk
$S(n)$ has infinite second moments. 
More precisely, we shall assume that $S(n)$ is asymptotically stable of index $\alpha<2$ 
when we study large deviations for local probabilities 
and asymptotics for the Green function. 
The asymptotics for the Green function of walks with finite variances have already been  studied in the literature: (a) Uchiyama~\cite{U2014} has considered lattice walks in a half space; (b) Duraj et al~\cite{DRTW18} have derived asymptotics of Green functions for a wider class of convex cones.
It is worth mentioning that the authors of~\cite{DRTW18} analyse first the case of a half space and use the estimates for the Green function for a half space in the subsequent analysis of convex cones. This fact underlines the importance of the case of half spaces. 
    
    We will say that $S(n)$ belongs to the domain of attraction of a multivariate stable law, if 
    \begin{equation}\label{eq:domain.attraction}
        \frac{S(n)}{c_n}\stackrel{d}\to \zeta_\alpha,
    \end{equation}
    where $\alpha\in(0,2]$ and $\zeta_\alpha$ has a  multivariate stable law of index $\alpha$. 
    Note that we assume that $S(n)$ is already centred. 
This does not restrict generality, when $\alpha\neq 1$, as one 
can subtract the mean for $\alpha>1$ 
and the centering is not needed for $\alpha<1$. 
This assumption excludes, however, some walks with $\alpha=1$ from consideration. 

Necessary and sufficient conditions for the convergence in \eqref{eq:domain.attraction} are given in~\cite{Rvacheva62}. 
    When $\alpha\in(0,2)$, the convergence will take placeif $\pr(|X|>t)$ is regularly varying of index $-\alpha$ and there exists  a measure $\sigma$ on the unit sphere  $\mathbb S^{d-1} $ such that
    \[
        \frac{\pr\left(|X|>x, \frac{X}{|X|}\in A\right)}{\pr(|X|>x)}\to \sigma(A),
       \quad x\to \infty, 
    \]   
    for any measurable $A$ on $\mathbb S^{d-1}$. 
    We will write $X\in \mathcal  D(d, \alpha,\sigma)$ 
    when \eqref{eq:domain.attraction} holds, 
    where $\sigma$ stands for the above measure on the unit sphere. 
    Let $g_{\alpha,\sigma}$ be the density of $\zeta_\alpha$.

    For $\alpha\in(0,2)$ we will assume that 
    $\sigma(\mathbb S^{d-1}\cap \mathbb H^+)>0$
    and $\sigma(\mathbb S^{d-1}\cap\mathbb H^-)>0$, 
    where $\mathbb H^- = \R^d\setminus \mathbb H^+$. 
    This  assumption implies that the first coordinate $X_1$ belongs to the one-dimensional domain of attraction. 
Moreover, let
\begin{equation*} 
\mathcal{A}:=\{0<\alpha <1;\,|\beta |<1\}\cup \{1<\alpha <2;|\beta |\leq
1\}\cup \{\alpha =1,\beta =0\}\cup \{\alpha =2,\beta =0\}
\end{equation*}%
be a subset in $\mathbb{R}^{2}.$ 
For $(\alpha ,\beta )\in \mathcal{A}$ and a 
random variable $X_1$ write $X_1\in \mathcal{D}\left( \alpha ,\beta \right) $ if
the distribution of $X_1$ belongs to the domain of attraction of a stable law
with characteristic function%
\begin{equation}
G_{\alpha ,\beta }\mathbb{(}t\mathbb{)}:=\exp \left\{ -c|t|^{\,\alpha
}\left( 1-i\beta \frac{t}{|t|}\tan \frac{\pi \alpha }{2}\right) \right\}
=\int_{-\infty }^{+\infty }e^{itu}g_{\alpha ,\beta }(u)du,\ c>0,  \label{std}
\end{equation}%
and, in addition, $\e[X]=0$ if this moment exists.
Then, $X_1\in \mathcal D(\alpha,\beta)$.

We will consider the case when $S_1(n)$ is oscillating, 
that is when $\pr(\tau<\infty)=1$ and $\e[\tau]=\infty$. 
Recall 
that the random walk $S_1(n)$ oscillates if and only if 
\begin{equation*}
    \sum_{n=1}^{\infty }\frac{1}{n}
    \mathbf{P}(S_1(n)>0)
    =\sum_{n=1}^{\infty }\frac{%
    1}{n}\mathbf{P}(S_{1}(n)\leq 0)=\infty .
\end{equation*}%
Rogozin~\cite{Rog71} investigated properties of $\tau$ and demonstrated
that the Spitzer condition
\begin{equation}
    n^{-1}\sum_{k=1}^{n}
    \mathbf{P}\left( S_1(k)>0\right) \rightarrow \rho \in
\left( 0,1\right) \quad \text{as }n\rightarrow \infty  \label{Spit}
\end{equation}%
holds if and only if $\tau$ belongs to the domain of attraction of a
positive stable law with parameter $\rho $. 
In particular, if $X_1\in \mathcal{%
D}(\alpha ,\beta )$ then (see,$\ $\ for instance, \cite{Zol57}) condition (%
\ref{Spit}) holds with
\begin{align}
\nonumber
\displaystyle\rho
&:=\sigma(\mathbb{H^+}\cap\mathbb{S}^{d-1})\\
&=\int_{0}^{\infty }g_{\alpha ,\beta }(u)du=\left\{
\begin{array}{ll}
\frac{1}{2},\ \alpha =1, &  \\
\frac{1}{2}+\frac{1}{\pi \alpha }\arctan \left( \beta \tan \frac{\pi \alpha
}{2}\right) ,\text{ otherwise}. &
\end{array}%
\right.  \label{ro}
\end{align}

We will further assume that the Spitzer-Doney condition holds 
\begin{equation}\label{eq:Doney-Spitzer}
    \pr(S_1(n)>0)\to \rho,\quad n\to\infty, 
\end{equation}
which is known to be equivalent to~\eqref{Spit}.  

The scaling sequence $\{c_n\}$  can be defined as follows, 
see~\cite{Rvacheva62}.  
Denote $\mathbb{Z}:=\left\{ 0,\pm 1,\pm 2,..\right\} ,$ $\mathbb{Z}%
_{+}:=\left\{ 1,2,..\right\} $ and let $\left\{ c_{n},n\geq 1\right\} $ be a
sequence of positive numbers specified by the relation%
\begin{equation}
c_{n}:=\inf \left\{ u\geq 0:\mu (u)\leq n^{-1}\right\} ,  \label{Defa}
\end{equation}%
where
\begin{equation*}
\mu (u):=\frac{1}{u^{2}}\int_{-u}^{u}x^{2}\mathbf{P}(|X|\in dx).
\end{equation*}%
It is known (see, for instance,~\cite[Ch. XVII, \S 5]{FE}) that for every $%
X\in \mathcal{D}(d,\alpha ,\sigma )$ the function $\mu (u)$ is regularly
varying with index $(-\alpha )$. This implies that $\left\{ c_{n},n\geq
1\right\} $ is a regularly varying sequence with index $\alpha ^{-1}$, i.e.
there exists a function $l_{1}(n),$ slowly varying at infinity, such that
\begin{equation}
c_{n}=n^{1/\alpha }l_{1}(n).  \label{asyma}
\end{equation}%
Then, convergence~\eqref{eq:domain.attraction} 
holds with this sequence $\{c_n\}$. 
In addition, the scaled sequence 
$\frac{S_1(n)}{c_{n}}$
converges in distribution, as $n\rightarrow\infty ,$ to the stable law
given by (\ref{std}). 

In one-dimensional case the study of asymptotics~\eqref{eq:green.n} was initiated  in~\cite{VW2009}, 
where  normal and small deviations of $p_n(0,y)$ were considered. 
Asymptotics for $p_n(x,y)$ with a 
 general starting point $x$ was studied then in~\cite{CaravennaChaumont2013} and~\cite{Doney2012}. 
Our assumption on $X_{1}$ is the same as in these papers and 
we use a similar approach for small and normal deviations. 
We used a different approach to study 
large deviations  in the 
multidimensional case. 
Large deviations seem to be the most complicated part 
of the present paper.

    As the first coordinate plays a distinctive role, 
    we will adopt the following notation. 
    For  $X(n)$ we will write $X(n)=(X_{1}(n),X_{(2,d)}(n))$, where $X_{1}(n)$ 
    corresponds to the first coordinate and $X_{(2,d)}(n)$ corresponds to  the remaining coordinates. 
    Similarly  we write 
    $S(n)=(S_{1}(n),S_{(2,d)}(n)), n=0,2,\ldots$.

    The following conditional limit theorem will be crucial for the rest of this
    article. The weak convergence in this theorem can be proved 
    similarly to~\cite{Doney1985}.  Existence 
    of the density is shown in Theorem~\ref{NormalDev'} below, 
    but can also be found similarly to  Remark~2 in~\cite{VW2009}. 
       \begin{theorem}
    \label{Cru} If
    $X\in \mathcal{D}(d,\alpha ,\sigma)$, then there exists a
     random vector $M_{\alpha ,\sigma}$ on $\mathbb H^+$ with density $p_{M_{\alpha ,\sigma}}(v)$ such that, for all $v\in \R^d$,
    \begin{equation}
    \lim_{n\rightarrow \infty }\pr\left( \frac{S(n)}{c_{n}}\in u+\Delta \mid \tau>n\right) 
    =\mathbf{P}(M_{\alpha ,\sigma}\in u+\Delta)
    =\int_{u+\Delta}p_{M_{\alpha ,\sigma}}(v)dv.
    \label{meander1}
    \end{equation}
Moreover, for every bounded and continuous function $f$,
$$
\e\left[f\left(\frac{S(n)}{c_n}\right)|\tau_x>n\right]
\to \e[f(M_{\alpha,\sigma})]
$$
uniformly in $x$ with $x_1\le\delta_nc_n$, $\delta_n\to0$.
    \end{theorem}


Our first result is an analogue of the classical local limit theorem, 
which is an extension of Theorems~3 and 5 in~\cite{VW2009} when $x=0$ 
and extends~\cite{Doney2012} and~\cite{CaravennaChaumont2013} 
for arbitrary starting point $x$.
\begin{theorem}\label{NormalDev'} 
    Suppose $X\in \mathcal{D}(d,\alpha ,\sigma)$. If the
distribution of $X$ is non-lattice then, for every $r >0$,
uniformly in $x$ with $x_1\le\delta_nc_n$, $\delta_n\to0$,
\begin{equation}
\sup_{y\in \mathbb{H}^+}\left|c_{n}^d\pr\left(x+S(n)\in y+r\Delta \mid \tau_x>n\right)-r^d {p_{M_{\alpha ,\sigma}}\left(\frac{y-x}{c_{n}}\right)}\right|\rightarrow 0.   \label{ND'}
\end{equation}
If the distribution of $X$ is lattice and if $\Z^d$ is the minimal lattice for $X$ then,
uniformly in $x\in \mathbb{H}^+\cap \Z^d$ with $x_1\le\delta_nc_n$, $\delta_n\to0$,
\begin{equation}
\sup_{y\in \mathbb{H}^+\cap \Z^d}\left|c_{n}^d\pr\left(x+S(n)= y \mid \tau_x>n\right)-{p_{M_{\alpha ,\sigma}}\left(\frac{y-x}{c_{n}}\right)}\right|\rightarrow 0.   \label{ND''}
\end{equation}
\end{theorem}

If the ratio $y/c_{n}$ varies with $n$ in such a way that $y_1/c_{n}\in
(b_{1},b_{2})$ for some $0<b_{1}<b_{2}<\infty $
and $|y_{(2,d)}|=O(c_n)$, we can rewrite (\ref{ND'}) 
as
\begin{equation*}
c_{n}^d\mathbf{P}(S(n)\in y+r\Delta \mid \tau>n)\sim r^d {%
p_{M_{\alpha,\sigma}}(y/c_{n})}\quad \text{as }n\rightarrow \infty .
\end{equation*}%
However, if $y_1/c_{n}\rightarrow 0$, then, in view of
\begin{equation*}
    \lim_{z_1\downarrow 0}{p_{M_{\alpha,\sigma}}(z)=0},
\end{equation*}%
relation (\ref{ND'}) gives only
\begin{equation}
c_{n}^d\pr(S(n)\in y+\Delta \mid \tau>n)=o\left( 1\right)
\quad \text{as }n\rightarrow \infty .  \label{Out1}
\end{equation}

Our next theorem refines (\ref{Out1}) in the mentioned domain of small
deviations, i.e. when $y_1/c_{n}\rightarrow 0.$ 
Let 
\[
    \tau^+:=\min\{k\ge1\colon S(k)\in\mathbb{H}^+\}=\min\{k\ge 1\colon S_{1}(k)>0\}. 
\]
Let $\chi ^{+}:=S_{1}(\tau ^{+}) (\chi ^{-}:=-S_{1}(\tau))$ be the first ascending(descending) ladder height and 
let $(\chi^+_n)_{n=1}^\infty ((\chi_n^-)_{n=1}^\infty )$ be  a sequence of i.i.d. copies of  $\chi^+ (\chi^-)$. 
Let 
\begin{align}
H(u):=\mathrm{I}\{u>0\}+\sum_{k=1}^{\infty }\mathbf{P}(\chi _{1}^{+}+\ldots
+\chi _{k}^{+}<u)  \label{ren}
\\ 
V(u):=\mathrm{I}\{u\ge 0\}+\sum_{k=1}^{\infty }\mathbf{P}(\chi _{1}^{-}+\ldots
+\chi _{k}^{-}\le u)\label{ren.desc}
\end{align}%
be the renewal function of the ascending (descending) ladder height process.
Clearly, $H$ is a left-continuous function.

\begin{theorem}\label{thm:small.deviations.arbitrary.lattice} 
    Suppose $X\in \mathcal{D}(d,\alpha ,\sigma)$. If the distribution of $X$ is lattice and if $\mathbb{Z}^d$ is the minimal lattice, then
    \[
        \pr(x+S(n)= y; \tau_x>n)\sim 
        V(x_1)H(y_1)\frac{g\left(0,\frac{y_{2,d}-x_{2,d}}{c_n}\right)}{nc_n^d}
        \]
        uniformly in $x,y \in \mathbb H^+\cap \mathbb Z^d$with $x_1,y_{1}\in (0,\delta _{n}c_{n}]$ such that $|x-y|\le A c_n $, where $\delta _{n}\rightarrow 0$
        as $n\rightarrow \infty $ and  $A$ is a fixed constant.
        
        If the
        distribution of $X$ is non-lattice then
            \[
                \pr(x+S(n)\in y+\Delta; \tau_x>n)\sim 
                V(x_1)\int_{y_1}^{y_1+1}H(u)du\frac{g\left(0,\frac{y_{2,d}-x_{2,d}}{c_n}\right)}{nc_n^d}
                \]
                uniformly in $x_1,y_{1}\in (0,\delta _{n}c_{n}]$ such that $|x-y|\le A c_n $, where $\delta _{n}\rightarrow 0$
                as $n\rightarrow \infty $ and  $A$ is a fixed constant. \end{theorem}

    To obtain the asymptotics for the Green function of $S(n)$ killed at leaving $\mathbb{H}^+$ 
    one has to estimate probabilities of local large deviations for $S(n)$.
    To this end we assume that there exists a regularly 
    varying $\phi$ of index $-\alpha$ 
    such that 
                \begin{equation}\label{eq:local.heavy}
                    \pr(X\in x+\Delta)\le \frac{\phi(|x|)} {|x|^{d}}=: g(|x|).
                \end{equation}
                and 
                \begin{equation}\label{eq:local.heavy.1}
                    a_1\phi(t)\le \pr(|X|> t)\le a_2\phi(t),\quad t>0, 
                \end{equation}
                for some positive constants $a_1,a_2$. 
    The fact that   global assumptions might in general give different 
    asymptotics  
    if the tails are too heavy is known in the multidimensional case 
    since Williamson~\cite{Williamson68} 
    who constructed a counterexample 
    for the Green functions on the whole space, which 
    has a different asymptotic behaviour without any local assumptions. 
                
    We now ready to formulate our bound for local large deviations.
                \begin{theorem}\label{thm.large.deviations}
                        
                    If~\eqref{eq:local.heavy} 
                    and~\eqref{eq:local.heavy.1} hold with some 
                    $\alpha<2$ then 
                    \begin{equation}\label{eq:arb.heavy.main.local.alpha2}
                        p_n(x,y) \le C_0
                        g(|x-y|)
                    H(y_1+1)V(x_1+1). 
                    \end{equation}    
    \end{theorem}

Having asymptotics for $p_n(x,y)$ for all possible ranges (for small, normal and large deviations), one can easily derive asymptotics for the Green function.
We start with the lattice case, where 
we combine Theorem~{\ref{thm:small.deviations.arbitrary.lattice} 
and Theorem~\ref{thm.large.deviations}  
to obtain the asymptotics of the Green function near the boundary.  

\begin{theorem}\label{thm.green.near.boundary}
    Assume $X\in\mathcal D(d,\alpha,\sigma)$ with some $\alpha<2$. 
    Suppose there exists a regularly varying  $\phi$ 
                    such that~\eqref{eq:local.heavy} 
                    and~\eqref{eq:local.heavy.1} hold.   
    \begin{enumerate}  
    \item
     If the distribution of $X$ is lattice and if $\Z^d$ is minimal for $X$               
then we have
    \[ 
        G(x,y) 
        \sim  C  \frac{H(y_1)V(x_1)}{|x-y|^d}  
        \int_0^{\infty} 
        g_{\alpha,\sigma}\left(0, \frac{y_{2,d}-x_{2,d}}{|x-y|} t\right) 
        t^{d-1} dt   
    \]
    for $x_1,y_1=o(|x-y|)$. 
    In particular,  in the isotopic case,  
    that is when the limiting density $\sigma$ is uniform on the unit sphere,
    \[
        G(x,y) \sim C_{\alpha} \frac{H(y_1)V(x_1)}{|x-y|^d}, \quad |x-y|\to \infty,     
    \]
    for $x_1,y_1=o(|x-y|)$. 
  \item  
    If $X$ is non-lattice then
     
    \[ 
        G(x,y) 
        \sim  C  \frac{\int_{y_1}^{y_1+1}H(u)duV(x_1)}{|x-y|^d}  
        \int_0^{\infty} 
        g_{\alpha,\sigma}\left(0, \frac{y_{2,d}-x_{2,d}}{|x-y|} t \right) 
        t^{d-1} dt   
    \]
    for $x_1,y_1=o(|x-y|)$. 
    In particular,  in the isotopic case,  
    that is when the limiting density $\sigma$ is uniform on the unit sphere,
    \[
        G(x,y) \sim C \frac{\int_{y_1}^{y_1+1}H(u)duV(x_1)}{|x-y|^d}, \quad |x-y|\to \infty,     
    \]
    for $x_1,y_1=o(|x-y|)$. 
\item
    In addition, there exists a constant $C$ such that 
    for all $x,y\in\mathbb{H}^+$, 
    \begin{equation}\label{eq.g.upper.bound}
      G(x,y) \le   C \frac{H(y_1)V(x_1)}{|x-y|^d}.
    \end{equation}
    \end{enumerate}
\end{theorem}  

\begin{remark} 
    For stable L\'evy processes 
    an exact formula (which can be analysed asymptotically) 
    for the Green function $g(x,y)$ was obtained in~\cite{TT2008}. 
    We are not aware of any result of this kind 
    for asymptotically stable random walks when $\alpha<2$. 
\end{remark}
\begin{remark}
    As we have already mentioned, walks with finite variances, which are a particular case of asymptotically stable walks with $\alpha=2$, were considered in~\cite[Theorem 2]{DRTW18} and~\cite{U2014} in the lattice case. 
In~\cite{DRTW18}  estimates 
of the behaviour of the Green function 
relied on~\cite{DW15,DuW15} and 
were obtained in a more general situation of convex cones. 
Using our methods we have  obtained  
asymptotics for the Green function in a half space 
for all asymptotically stable walks with $\alpha=2$.
Specialising this result to the case of finite variances,
we can obtain asymptotics under weaker moment conditions than 
in~\cite{DRTW18} and 
stronger than in~\cite{U2014}. 
The method we use is different from~\cite{U2014} 
who approached the problem using the potential kernel.

The only difference between $\alpha<2$ and $\alpha=2$ are estimates for local large deviations. In the case $\alpha=2$ one has to take care of the Gaussian component, 
what leads to very long calculations. For that reason we have decided not to include walks with $\alpha=2$ in the present paper and to consider 
this case in a separate paper. 

\end{remark}    

\begin{remark} 
It seems to be possible to extend 
the estimates for the Green function 
for asymptotically isotropic random walk in cones. 
This will be considered elsewhere 
and should also allow one to extend 
the results of~\cite{DW10,DW15,DW19,DW21,DuW15,DRTW18} 
to the stable case. 

\end{remark}

Since exit times from a half space can be considered as exit times for one-dimensional random walks, it is quite natural to use the methods, which are typical for walks on the real line.
In the proofs of our results on the asymptotic behaviour of the probabilities we follow this strategy and use a lot 
methods from \cite{Doney2012} and \cite{VW2009}. 
However is is worth mentioning that additional dimensions cause many additional technical problems, 
because of  the possibility of big jumps of random 
the walk which are still close to the boundary hyperplane $\{x:x_1=0\}$. 
While in the finite variance case one can try 
to control these jumps by assuming existence of  additional moments, 
this is not possible in the infinite variance case.  
For these reasons our estimates for $p_n(x,y)$ can not be considered as a straight forward generalisation of one-dimenisonal results.

\section{Preliminary upper bounds}
In this section we find bounds for $\pr(\tau_x>n)$.
Since $\tau_x$ is actually a stopping time for the one-dimensional walk $S_1(k)$, we may apply Lemma 3 from~\cite{DW16}, which gives us the following estimate. 
\begin{lemma}
\label{thm:tail_bound_1}
Assume that \eqref{eq:Doney-Spitzer} holds. 
Then there exists $C_0$ such that for every  
$x\in \mathbb H^+$ one has 
\begin{equation} 
\label{eq:Tg_bound.2}
\frac{\pr(\tau_x>n)}{\pr(\tau_0>n)}\le C_0V(x_1),\quad n\ge1.
\end{equation}
\end{lemma}

Next lemma is an extension of~\cite[Lemma~20]{VW2009} to the case of half spaces.
We will give a  proof following a different approach, 
which relies on Lemma~\ref{thm:tail_bound_1}. 
This proof works in one-dimensional case as well, thus simplifying the corresponding arguments of~\cite{VW2009}. 

For $y=(y_1,\ldots, y_d),z=(z_1,\ldots, z_d)\in \R^d$  
    we will write $y\le z$ if $y_k\le z_k$ for all $1\le k\le d$. 

\begin{lemma}\label{lem:lem20}
    Assume that $X\in\mathcal{D}(d,\alpha,\sigma)$. 
    Then, there exists $C>0$ such that for all
    $x,y\in \mathbb H^+$ and all $n\ge 1$ we have
    \begin{equation}\label{eq:49}
        p_n(x,y) \le \frac{CV(x_1)H(y_1)}{n c_n^d}.
    \end{equation}
    Similar result holds for the stopping time $\tau^+$:
    $$
    \pr(S(n)\in x+\Delta;\tau^+>n)\le C\frac{V(x_1)}{nc_n^d}.
    $$
\end{lemma}

\begin{proof}
We prove the first statement only. The proof of the second estimate
requires only notational changes.

    Put $n_1=[n/4], n_2=[3n/4]-n_1, n_3=n-[3n/4]$. 
    We split the probability of interest into 3 parts, 
    \begin{align*}
        p_n(x,y)&=
        \int_{\mathbb H^+} \pr(\tau_x>n_1, x+S(n_1)\in du)
        \int_{\mathbb H^+}\pr(\tau_u>n_2, u+S(n_2)\in dz)\\
    &\hspace{1cm}\times        \pr(\tau_z>n_3, z+S(n_3)\in y+\Delta). 
    \end{align*}
    Now we will make  use of the time inversion. 
    Let 
    $$\widetilde X(n)=-X(1),
    \widetilde X(n-1)=-X(2),\ldots, 
    \widetilde X(1)=-X(n)
    $$ 
    and
    \begin{align*} 
        \widetilde S(k) 
        &= \widetilde X(1)+\cdots+\widetilde X(k)\\ 
        &=-X(n)-\cdots-X(n-k+1) = 
        S(n-k)-S(n), \quad k=1,\ldots,n. 
    \end{align*}
    Let $1_d=(1,\ldots,1)$. 
    Then, 
    \begin{align*}
        &\pr(z+S(n_3)\in y+\Delta; \tau_z>n_3)\\
        &\hspace{1cm}= \pr(z+S(n_3)\in y+\Delta; z_1+
        \min(S_1(1),,\ldots, S_1(n_3)>0)\\
        &\hspace{1cm}= \pr(y+1_d+\widetilde S(n_3)\in z+(0,1]^d; 
        z_1+\min(\widetilde S_1(n_3-1),\ldots, \widetilde S_1(1))>\widetilde S_1(n_3))\\
        &\hspace{1cm}\le \pr(y+1_d+\widetilde S(n_3)\in z+(0,1]^d;  \widetilde \tau_{y+1_d}>n_3).        
    \end{align*}
Then, using the concentration function inequalities, we can continue as follows, 
\begin{align*}
&\int_{\mathbb H^+}\pr(\tau_u>n_2, u+S(n_2)\in dz)\pr(\tau_z>n_3, z+S(n_3)\in y+\Delta)\\
&\le \int_{\mathbb H^+}\pr(\tau_u>n_2, u+S(n_2)\in dz)\pr(y+1_d+\widetilde S(n_3)\in z+(0,1]^d, \widetilde \tau_{y+1_d}>n_3)\\        
&\le     
\sum_{j_1=0}^\infty \sum_{j_2=-\infty}^\infty 
\sum_{j_d=-\infty}^\infty 
\pr(\tau_u>n_2, u+S(n_2)\in [j_1,j_1+1)\times\dots [j_d,j_d+1) )\\
&\times 
\pr(y+1_d+\widetilde S(n_3)\in (j_1,\ldots,j_d)+2\Delta, \widetilde \tau_{y+1_d}>n_3)\\        
&\le \frac{C}{c_n^d} 
\sum_{j_1=0}^\infty \sum_{j_2=-\infty}^\infty 
\sum_{j_d=-\infty}^\infty 
\pr(y+1_d+\widetilde S(n_3)\in (j_1,\ldots,j_d)+2\Delta, \widetilde \tau_{y+1_d}>n_3)\\        
&\le \frac{C2^d}{c_n^d}  
\pr(\widetilde \tau_{y+1_d}>n_3).\\        
\end{align*}
As a result we obtain the bound 
\[
    p_n(x,y) \le \frac{C}{c_n^d}\pr(\tau_x>n_1)
    \pr(\widetilde \tau_{y+1_d}>n_3). 
\]     
Applying Lemma~\ref{thm:tail_bound_1} we obtain 
\begin{align*}
    p_n(x,y) &\le \frac{CH(y_1)}{c_n^d}
    \pr(\tau_x>n_1)
    \pr(\widetilde \tau_0>n_3)\\
    &\le \frac{CH(y_1)V(x_1)}{c_n^d}
    \pr(\tau_0>n_1)
    \pr(\tau^+_0>n_3).  
\end{align*}
Here recall that one can deduce by Rogozin's result that (\ref{Spit}) holds if and only if
    there exists a function $l(n)$, slowly varying at infinity, such that, as $%
    n\rightarrow \infty $,
    \begin{equation}\label{integ1}
    \mathbf{P}\left( \tau>n\right) \sim \frac{l(n)}{n^{1-\rho }},\ \
    \mathbf{P}\left( \tau ^{+}>n\right) \sim \frac{1}{\Gamma (\rho )\Gamma
    (1-\rho )n^{\rho }l(n)}.  
    \end{equation}%
Then we obtain that  $\pr(\tau_0>n) \pr(\tau^+_0>n)\sim \frac{C}{n}$
and  arrive at the conclusion. 
\end{proof}

    \begin{lemma}\label{lem:lem19}
        Assume that the random walk $S(n)$ is asymptotically stable. 
        Then, there exists $C>0$ such that for $x,y\in \mathbb H^+$ and all $n\ge 1$
        \begin{equation}\label{eq:44}
            p_n(x,y) \le \frac{C V(x_1) }{c_n^d} \frac{l(n)}{n^{1-\rho}}.
        \end{equation}
    \end{lemma}
    \begin{proof}
        For $n\ge 2$,
        \begin{align*}
            p_n(x,y)&\le \pr(\tau_x>n/2)\sup_{z\in \R^d}\pr(S_{[n/2]} \in z+\Delta).
        \end{align*}
        Applying now a concentration function inequality we obtain 
        \[
            p_n(x,y) \le \pr(\tau_x>n/2) \frac{C_1}{c_n^d} \le  
            \frac{CV(x_1)}{c_n^d}\frac{l(n)}{n^{1-\rho}},
            \]
         since 
         \[
            \pr(\tau_x>n) \le C V(x_1)  \pr(\tau>n)\le  
            C V(x_1)\frac{l(n)}{n^{1-\rho}}. 
        \]
    \end{proof}
     
    Before proving the next lemma recall the following result, see~\cite[Lemma 13]{VW2009}. 
\begin{lemma}
    \label{Renew2} Suppose $X_{1}\in \mathcal{D}(\alpha ,\beta )$. 
    Then, as $%
    u\rightarrow \infty $,%
    \begin{equation}
    H(u)\sim \frac{u^{\alpha \rho }l_{2}(u)}{\Gamma (1-\alpha \rho )\Gamma
    (1+\alpha \rho )}  \label{RenStand}
    \end{equation}%
    if $\alpha \rho <1$, and
    \begin{equation}
    H(u)\sim ul_{3}(u)  \label{RenewRelat}
    \end{equation}%
    if $\alpha \rho =1,$ where%
    \begin{equation*}
    l_{3}(u):=\left( \int_{0}^{u}\pr\left( \chi ^{+}>y\right) dy\right)
    ^{-1},\text{ \ }u>0.
    \end{equation*}%
    In addition, there exists a constant $C>0$ such that, in both cases,
    \begin{equation}
    H(c_{n})\sim Cn\mathbf{P}(\tau>n)\quad \text{as }n\rightarrow \infty .
    \label{AsH}
    \end{equation}
    \end{lemma}

    \begin{lemma}
        \label{LDensity} There exists a constant $C\in \left( 0,\infty \right) $
        such that, for all $z\in \lbrack 0,\infty )\times \R^{d-1},$%
        \begin{equation*}
        \lim \sup_{\varepsilon \downarrow 0}\varepsilon ^{-d}\mathbf{P}\left(
        M_{\alpha ,\sigma }\in z+\varepsilon\Delta \right) \leq C\min
        \{1,(z_1)^{\alpha \rho }\}.
        \end{equation*}%
        In particular,
        \begin{equation*}
        \lim_{z_1\downarrow 0}\lim \sup_{\varepsilon \downarrow 0}\varepsilon ^{-d}%
        \mathbf{P}\left( M_{\alpha ,\sigma }\in \lbrack z,z+\varepsilon )\right) =0.
        \end{equation*}
        \end{lemma}
        
        \begin{proof}
        For all $z\in \lbrack 0,\infty )\times \R^{d-1},$ 
        and all $\varepsilon >0$ we have
        \begin{equation*}
        \pr\left( M_{\alpha ,\sigma }\in z+\varepsilon \Delta \right) \leq
        \lim \sup_{n\rightarrow \infty }\mathbf{P}\left( S(n)\in 
        c_{n} z+\varepsilon  c_n\Delta \mid \tau>n\right) .
        \end{equation*}%
        Applying (\ref{P11}) gives
        \begin{equation*}
            \pr\left( S(n)\in 
            c_{n} z+\varepsilon  c_n\Delta \mid \tau>n\right)
        \leq C\frac{H(\min (c_{n},(z_1+\varepsilon )c_{n}))%
        }{nc_{n}^d\mathbf{P}(\tau>n)}(\varepsilon c_{n})^d.
        \end{equation*}%
        Recalling that $H(x)$ is regularly varying with index $\alpha \rho $ by
        Lemma \ref{Renew2} and taking into account (\ref{AsH}), we get
        \begin{align*}
            \pr\left( S(n)\in 
            c_{n} z+\varepsilon  c_n\Delta \mid \tau>n\right) 
            & \leq C\varepsilon^d \min \{1,(z_1+\varepsilon )^{\alpha \rho }\}%
        \frac{H(c_{n})}{n\mathbf{P}(\tau>n)} \\
        & \leq C\varepsilon^d \min \{1,(z_1+\varepsilon )^{\alpha \rho }\}.
        \end{align*}%
        Consequently,
        \begin{equation}
        \pr\left( M_{\alpha ,\sigma }\in z+\varepsilon \Delta \right)\leq C\varepsilon^d
        \min \{1,(z_1+\varepsilon )^{\alpha \rho }\}.  \label{Majorante}
        \end{equation}%
        This inequality shows that there exists a constant $C\in (0,\infty )$ such
        that
        \begin{equation*} 
        \lim \sup_{\varepsilon \downarrow 0}\varepsilon ^{-d}
        \pr\left( M_{\alpha ,\sigma }\in z+\varepsilon \Delta \right) \leq C\min
        \{1,(z_1)^{\alpha \rho }\}, \text{ for all \ }z\in \lbrack 0,\infty )\times \R^{d-1}, 
        \end{equation*}%
        as desired.
\end{proof}

    \begin{corollary}
        \label{Cmin}
        Let $X\in \mathcal{D}(d,\alpha ,\sigma).$ 
        There exists a
        constant $C\in \left( 0,\infty \right) $ such that, for all 
        $n\geq 1$ and $x,y\in\mathbb H^+$,%
        \begin{equation}
        p_{n}(x,y)\leq CV(x_1)\frac{H(\min (c_{n},y_1))}{nc_{n}^d}.  \label{P11}
        \end{equation}%
        \end{corollary}
        
        \begin{proof}
        The desired estimates follow from (\ref{AsH}) and 
        Lemmata~\ref{lem:lem20} and~\ref{lem:lem19}. 
        \end{proof}

\section{Baxter-Spitzer identity}
    We will need the following multidimensional extension of one-dimensional Baxter-Spitzer identity, see~\cite[Lemma 3.2]{TT2002}.
    \begin{lemma}\label{lem:tanaka}
        For $t\in \R^d$ and $|s|<1$  the following identity
        \[
            1+\sum_{n=1}^\infty s^n \e[e^{it\cdot S(n)};\tau_0>n]
            =\exp\left\{
                \sum_{n=1}^\infty    \frac{s^n}{n}\e\left[e^{it\cdot S(n)}; S_1(n)>0\right]
            \right\}. 
        \]
    \end{lemma}

    We will now  follow closely~\cite{VW2009}. 
    Put 
    \[
        B_n(y):=\pr(S(n)< y; \tau_x>n).  
    \]
    and $b_n(y):= p_n(0,y)$. 
    Lemma~15 of~\cite{VW2009} extends as follows 
    \begin{lemma}\label{lem:lema15}
        The sequence of functions $\{\, B_n(y), n\ge 1\,\}$  satisfies the recurrence equations 
        \begin{equation}\label{eq:40}
            nB_n(y) = \pr(S(n)<y, S_1(n)>0)
            +\sum_{k=1}^{n-1}
            \int_{\R^{d}}
            \pr(S(k)<y-z, S_1(k)>0) dB_{n-k}(z)
        \end{equation}
        and 
        \begin{equation}\label{eq:41}
            nB_n(y) = \pr(S(n)<y, S_1(n)>0)
            +\sum_{k=1}^{n-1}
            \int_{\R^{d}}
            B_{n-k}(y-z)
            \pr(S(k)\in dz, S_1(k)>0). 
        \end{equation}
    \end{lemma}
    The proof is analogous to the proof of Lemma~15 of~\cite{VW2009}.
    
    To deal with random walks started at  an arbitrary point 
    we will prove 
    Lemma~\ref{lem.arbitrary.starting point} below that 
    extends (17) in~\cite{Doney2012}.  
    Put 
    \[
        p_n^+(0,dy):=\pr(S(n)\in dy, \tau^+>n).
    \]
    We will slightly abuse the notation and write 
    \[
        p_n(x,dy)=\pr(x+S(n)\in dy, \tau_x>n).
    \]

    \begin{lemma}\label{lem.arbitrary.starting point}
    For $x\in\overline{\mathbb H^+} ,y \in \mathbb H^+$ we have 
    \begin{align}
        \label{eq:g.arbitrary.starting.point}
    p_n(x,dy)&= \sum_{k=0}^n  
    \int_{(0,x_1\wedge y_1 ] \times \R^{d-1}}
    p_k^+(0, dz-x)
        p_{n-k}(0, dy-z)
    \end{align}
    and for $x\in\overline{\mathbb H^+} \cap \Z ,y \in \mathbb H^+\cap \Z$
    \begin{align}
        \label{eq:g.arbitrary.starting.point.interval}
    p_n(x,y)&=\sum_{k=0}^n 
        \int_{(0,x_1\wedge (y_1+1) ] \times \R^{d-1}}
        p_k^+(0,dz-x)b_{n-k}(y-z).
    \end{align}
    \end{lemma}    
    \begin{proof}
    Decomposing the trajectory of the walk at the minimum of the first coordinate and using the duality lemma for random walks, we get
    \begin{align*}
        p_n(x,dy)&=\sum_{k=0}^n 
        \int_{(0,x_1\wedge y_1 ] \times \R^{d-1}}
        \pr(x+S(k)\in dz, S_1(k)\le \min_{j\le k} S_{1}(j))\\ 
        &\hspace{1cm}\times 
        \pr(z+S(n-k)\in dy, \tau_0>n-k)\\
        &=\sum_{k=0}^n 
        \int_{(0,x_1\wedge y_1 ] \times \R^{d-1}}
        \pr(x+S(k)\in dz, \tau^+>k)\\ 
        &\hspace{1cm}\times 
        \pr(z+S(n-k)\in dy, \tau_0>n-k)\\
        &=\sum_{k=0}^n 
        \int_{(0,x_1\wedge y_1 ] \times \R^{d-1}}
        p_k^+(0, dz-x)
        p_{n-k}(0, dy-z).
    \end{align*}     
    Integrating~\eqref{eq:g.arbitrary.starting.point},
    we obtain the second equality~\eqref{eq:g.arbitrary.starting.point.interval}.  
    \end{proof}

\section{Probabilities of normal deviations: proof of Theorem~\ref{NormalDev'} for $x=0$.}\label{sec:normal}
Let $H_{y_1}^+ = \{(z_1,z^{(2,d)}): 0<z_1<y_1 )\}$. 
It follows from~\eqref{eq:40} that 
\begin{align}
    \nonumber 
nb_n(y)&=\pr(S(n)\in y+\Delta)\\ 
&\hspace{1cm}+\sum_{k=1}^{n-1} 
\int_{\R^d}\pr(S(k)\in y-z+\Delta, S_1(k)>0) dB_{n-k}(z)
\nonumber\\
&:=R_{\varepsilon }^{(0)}(y)+ R_{\varepsilon }^{(1)}(y)+R_{\varepsilon }^{(2)}(y)+R_{\varepsilon
}^{(3)}(y),
\label{c1}
\end{align}
where, for any fix $\varepsilon \in (0,1/2)$ and with a slight abuse of
notation,
\begin{align*}
R_{\varepsilon }^{(1)}(y)&:=\sum_{k=1}^{\varepsilon n}
\int_{H_{y_1}^+}\pr(S(k)\in  y-z+\Delta, S_1(k)>0)dB_{n-k}(z),\\
R_{\varepsilon }^{(2)}(y)&:=\sum_{k=\varepsilon n}^{(1-\varepsilon
)n}\int_{H_{y_1}^+}\pr(S(k)\in  y-z+\Delta, S_1(k)>0)dB_{n-k}(z),\\ 
{R}_{\varepsilon }^{(3)}(y)&:=\pr({S}_{n}\in y+\Delta)+\sum_{k=[(1-\varepsilon )n]+1}^{n-1}
\int_{\R^d}\pr(S(k)\in  y-z+\Delta, S_1(k)>0)dB_{n-k}(z)
\end{align*}%
and 
\[
R_\varepsilon^{(0)}(y):= 
\sum_{k=1}^{(1-\varepsilon)n}
\int_{[y_1,y_1+1)\times \R^{d-1}}\pr(S(k)\in  y-z+\Delta, S_1(k)>0)dB_{n-k}(z).
\] 

Fix some $t\in\Z^d$. If $z\in(t+\Delta)$ then 
\begin{align*}
\left\{S(k)\in y-z+\Delta\right\}
&=\left\{S_j(k)\in[y_j-z_j,y_j-z_j+1)\text{ for all }j\right\}\\
&\subseteq\left\{S_j(k)\in[y_j-t_j-1,y_j-t_j+1)\text{ for all }j\right\}\\
&=\left\{S(k)\in y-t-1+2\Delta\right\}.
\end{align*}
Consequently,
$$
\pr(S(k)\in  y-z+\Delta, S_1(k)>0)
\le \pr(S(k)\in  y-t-1+2\Delta, S_1(k)>0)
$$
for all $z\in t+\Delta$.
Applying this estimate, we conclude that, for every $A\subset\R^d$,
\begin{align*}
&\int_A\pr(S(k)\in  y-z+\Delta, S_1(k)>0)dB_{n-k}(z)\\
&\le
\sum_{t\in\Z^d:(t+\Delta)\cap A\neq\emptyset}
\pr(S(k)\in  y-t-1+2\Delta, S_1(k)>0)b_{n-k}(t).
\end{align*}
Combining this estimate with Corollary~\ref{Cmin} with $x=0$, we obtain
\begin{align}
\label{eq:int_A}
\nonumber
&\int_A\pr(S(k)\in  y-z+\Delta, S_1(k)>0)dB_{n-k}(z)\\
&\le
\frac{C}{(n-k)c_{n-k}^d}
\sum_{t\in\Z^d:(t+\Delta)\cap A\neq\emptyset}
H(t_1\wedge c_{n-k})\pr(S(k)\in  y-t-1+2\Delta, S_1(k)>0).
\end{align}
In order to bound ${R}^{(0)}_\varepsilon(y)$ we apply \eqref{eq:int_A} with $A=[y_1,y_1+1)\times \R^{d-1}$:
\begin{align*}
    &{R}^{(0)}_\varepsilon(y)\\
    & \le    
    \sum_{k=1}^{(1-\varepsilon)n}\frac{C}{(n-k)c_{n-k}^d}
    \sum_{t\in\Z^d:t_1\in[y_1-1,y_1+1)}
     H(t_1\wedge c_{n-k})\pr(S(k)\in  y-t-1+2\Delta, S_1(k)>0)\\
    &\le 
    C_\varepsilon \frac{H(y_1\wedge c_{n})}{nc_n^d}
    \sum_{k=1}^{(1-\varepsilon)n}\pr(S_1(k)\in(0,2)).
\end{align*}
Noting that $\pr(S_1(k)\in(0,2))\to0$ as $k\to\infty$, we get 
$$
\frac{c_n^d}{H(c_n)}R^{(0)}_\varepsilon(y)\to0.
$$
Taking into account~\eqref{AsH}, we conclude that  
\begin{equation}
\limsup_{n\rightarrow \infty }\frac{c_{n}^d}{n\pr(\tau>n)}%
R^{(0)}_\varepsilon(y)=0.  \label{c3}
\end{equation}
Using the Stone theorem we obtain 
\begin{align*}
    {R}_{\varepsilon }^{(3)}(y)\le \frac{C}{c_n^d}
    \left(
        1+ \sum_{k=1}^{\varepsilon n} \pr(\tau>k)
    \right).
\end{align*}
Further, by (\ref{integ1}),
\begin{equation*}
\sum_{k=0}^{\varepsilon n}\pr(\tau>k)\sim \rho
^{-1}\varepsilon ^{\rho }n\pr(\tau>n)\quad \text{as }%
n\rightarrow \infty .
\end{equation*}%
As a result,  we obtain
\begin{equation}
\limsup_{n\rightarrow \infty }\frac{c_{n}^d}{n\pr(\tau>n)}%
\sup_{y\in \R^d}R_{\varepsilon }^{(3)}(y)\leq C\varepsilon ^{\rho }.  \label{n1}
\end{equation}
Applying \eqref{eq:int_A} with $A=[0,y_1)\times \R^{d-1}$, we get 
\begin{align*}
    {R}_{\varepsilon }^{(1)}(y)
    \le \frac{CH(c_n)}{nc_n^d}
    \sum_{k=1}^{\varepsilon n}
     \pr(S_1(k)\in (0,y_1+2))    
     \le \varepsilon \frac{CH(c_n)}{c_n^d}
\end{align*}
From this estimate and (\ref{AsH}) we deduce
\begin{equation}
\limsup_{n\rightarrow \infty }\frac{c_{n}^d}{n\mathbf{P}(\tau>n)}%
\sup_{y\in \mathbb H^+}R_{\varepsilon }^{(1)}(y)\leq C\varepsilon .  \label{n2}
\end{equation}

Thus, in the non-lattice case we combine the Stone local limit theorem with
the first equality in (\ref{meander1}) and obtain, uniformly in $y\in \mathbb H^+$, 
\begin{align*}
R_{\varepsilon }^{(2)}(y)& =\sum_{k=[\varepsilon n]+1}^{[(1-\varepsilon )n]}%
\frac{1}{c_{n-k}^d}\int_{0}^{y_1}\int_{\R^{d-1}}g_{\alpha,\sigma}\left( \frac{y-z}{c_{n-k}}%
\right) dB_{k}(z)+o\left( \frac{1}{c_{n\varepsilon }^{d}}\sum_{k=1}^{n}\pr(S_{1}(k)<y_1,\tau>k)%
\right)  \\
& =\sum_{k=[\varepsilon n]+1}^{[(1-\varepsilon )n]}\frac{\mathbf{P}(\tau
^{-}>k)}{c_{n-k}^d}\int_{0}^{y_1/c_k}\int_{\R^{d-1}}g_{\alpha,\sigma}\left( \frac{y-c_{k}u}{%
c_{n-k}}\right) \mathbf{P}(M_{\alpha,\sigma}\in du) \\
& \hspace{2cm}+o\left( \frac{1}{c_{n\varepsilon }^d}\sum_{k=1}^{n}\pr(S_{1}(k)<y_1,\tau>k)+%
\sum_{k=1}^{n-1}\frac{\mathbf{P}(\tau>k)}{c_{n\varepsilon }^d}\right) .
\end{align*}%
According to (\ref{integ1}),
\begin{equation*}
    \sum_{k=1}^{n}\pr(S^{(1)}<y_1,\tau>k)\leq \sum_{k=1}^{n}\pr(\tau>k)\leq Cn%
\mathbf{P}(\tau>n).
\end{equation*}%
Hence, 
\begin{align*}
    R_{\varepsilon }^{(2)}(y)
    & =\sum_{k=[\varepsilon n]+1}^{[(1-\varepsilon )n]}\frac{\mathbf{P}(\tau
    ^{-}>k)}{c_{n-k}^d}\int_{0}^{x_1/c_k}\int_{\R^{d-1}}g_{\alpha,\sigma}\left( \frac{x-c_{k}u}{%
    c_{n-k}}\right) \mathbf{P}(M_{\alpha,\sigma}\in du) \\
    & \hspace{2cm}+o\left( 
    \frac{n\mathbf{P}(\tau>n)}{c_{n\varepsilon }^d}\right) .
\end{align*}
Since $c_{k}$ and $\mathbf{P}(\tau>k)$ are regularly varying and $%
g_{\alpha,\sigma}(x)$ is uniformly continuous in $\R^d$, we
let, for brevity, $v=x/c_{n}$ and continue the previous estimates for $%
R_{\varepsilon }^{(2)}(y)$ with
\begin{align*}
    & =\frac{\mathbf{P}(\tau>n)}{c_{n}^d}\sum_{k=[\varepsilon
    n]+1}^{[(1-\varepsilon )n]}\frac{(k/n)^{\rho -1}}{(1-k/n)^{1/\alpha }}%
    \int_{0}^{v^{(1)}/(k/n)^{1/\alpha }}\int_{\R^{d-1}}
    g_{\alpha,\sigma}\left( \frac{%
    v-(k/n)^{1/\alpha }u}{(1-k/n)^{1/\alpha }}\right) \mathbf{P}(M_{\alpha,\sigma}\in du) \\
    & \hspace{5cm}+o\left( \frac{n\mathbf{P}(\tau>n)}{c^d_{n\varepsilon }}%
    \right)  \\
    & =\frac{n\mathbf{P}(\tau>n)}{c_{n}^d}f(\varepsilon ,1-\varepsilon
    ;v)+o\left( \frac{n\mathbf{P}(\tau>n)}{c^d_{n\varepsilon }}\right),
    \end{align*}%
where, for $0\leq w_{1}\leq w_{2}\leq 1$,
    \begin{equation}
    f(w_{1},w_{2};v):=\int_{w_{1}}^{w_{2}}\frac{t^{\rho -1}dt}{(1-t)^{1/\alpha }}%
    \int_{0}^{v^{(1)}/t^{1/\alpha }}\int_{\R^d}g_{\alpha,\sigma}\left( \frac{%
    v-t^{1/\alpha }u}{(1-t)^{1/\alpha }}\right) \mathbf{P}(M_{\alpha,\sigma}\in
    du).  \label{DenF}
    \end{equation}
Observe that, by boundedness of $g_{\alpha,\sigma}\left( y\right)$,
    \begin{equation*}
    f(0,\varepsilon ;v)\leq C\int_{0}^{\varepsilon }t^{\rho -1}dt\leq
    C\varepsilon ^{\rho }.
    \end{equation*}%
    Further, it follows from (\ref{Majorante}) that $\int \phi (u)\pr
    (M_{\alpha,\sigma}\in du)\leq C\int \phi (u)du$ for every non-negative
    integrable function $\phi $. Therefore,
    \begin{align*}
    & f(1-\varepsilon ,1;v) \\
    & \leq C\int_{1-\varepsilon }^{1}\frac{t^{\rho -1}dt}{(1-t)^{1/\alpha }}%
    \int_{0}^{v^{(1)}/t^{1/\alpha }}\int_{\R^{d-1}}g_{\alpha ,\sigma }\left( \frac{%
    v-t^{1/\alpha }u}{(1-t)^{1/\alpha }}\right) du=\left( z=\frac{v-t^{1/\alpha
    }u}{(1-t)^{1/\alpha }}\right)  \\
    & =C\int_{1-\varepsilon }^{1}t^{\rho -1-1/\alpha
    }dt\int_{0}^{v^{(1)}/(1-t)^{1/\alpha }}\int_{\R^{d-1}}g_{\alpha ,\sigma }\left(
    z\right) dz\leq C\varepsilon .
    \end{align*}%
    As a result we have
\begin{equation}
\limsup_{n\rightarrow \infty }\sup_{y\in \mathbb H^+}\left\vert \frac{c_{n}^d}{n\mathbf{P}%
(\tau>n)}R_{\varepsilon }^{(2)}(y)-f(0,1;y/c_{n})\right\vert \leq
C\varepsilon ^{\rho }.  \label{n3}
\end{equation}%
Combining (\ref{c3}) -- (\ref{n3}) with representation (\ref{c1}) leads to
\begin{equation}
\limsup_{n\rightarrow \infty }\sup_{y\in\mathbb H^+}\left\vert \frac{c^d_{n}}{\mathbf{P}%
(\tau>n)}b_{n}(y)-f(0,1;y/c_{n})\right\vert \leq C\varepsilon ^{\rho }.
\label{n4}
\end{equation}%
Since $\varepsilon >0$ is arbitrary, it follows that, as $n\rightarrow
\infty $,%
\begin{equation}
\frac{c_{n}^d}{\mathbf{P}(\tau>n)}b_{n}(y)-f(0,1;y/c_{n})\rightarrow 0
\label{Fconv}
\end{equation}%
uniformly in $y\in \mathbb H^+$. 
Recalling (\ref{meander1}), we deduce by integration of (%
\ref{Fconv}) and evident transformations that
\begin{equation}
\int_{u+r\Delta}f(0,1;z)dz=\mathbf{P}(M_{\alpha,\sigma}\in u+r\Delta)  \label{deniq}
\end{equation}%
for all $u\in\mathbb H^+, r>0$. This means, in particular, that the
distribution of $M_{\alpha,\sigma}$ is absolutely continuous.
Furthermore,
it is not difficult to see that $z\mapsto f(0,1;z)$ is a continuous mapping.
Hence, in view of (\ref{deniq}), we may consider $f(0,1;z)$ as a continuous
version of the density of the distribution of $M_{\alpha,\sigma}$ and let $%
p_{M_{\alpha,\sigma}}(z):=$ $f(0,1;z).$ 
This and (\ref{Fconv}) \ imply the
statement of Theorem \ref{NormalDev'} for $\Delta =[0,1)^d$. 
To establish the
desired result for arbitrary $r\Delta, r>0$ it suffices to consider the random
walk $S(n)/r $ and to observe that
\begin{equation*}
c_{n}^{r }:=\inf \left\{ u\geq 0:\frac{1}{u^{2}}\int_{-u}^{u}x^{2}\pr%
\left( \frac{|X|}{r }\in dx\right) \right\} =c_{n}/r^d.
\end{equation*}


\section{Probabilities of normal deviations when random walks start at an arbitrary starting point}
\begin{proof}[Proof of Theorem~\ref{Cru}]
Due to the shift invariance in any direction orthogonal to
$(1,0,\ldots,0)$ we may consider, without loss of generality, only the case when the random walk starts at $x=(x_1,0,\ldots,0)$ with some $x_1>0$.

As we have already mentioned before, repeating the arguments from \cite{Doney1985} one can easily show that $\pr(\frac{S(n)}{c_n}\in\cdot|\tau>n)$ and $\pr(\frac{S(n)}{c_n}\in\cdot|\tau^+>n)$ converge weakly. Recall also that the limit of $\pr(\frac{S(n)}{c_n})\in\cdot|\tau>n)$ is denoted by $M_{\alpha,\sigma}$.

Fix an arbitrary Borel set $A\subset\mathbb{H}^+$.
According to Lemma~\ref{lem.arbitrary.starting point},
\begin{align}
\nonumber
\label{eq:thm1.1}
&\pr\left(\frac{x+S(n)}{c_n}\in A;\tau_x>n\right)\\
&\hspace{1cm}=\sum_{k=0}^n\int_{(0,x_1]\times\R^{d-1}}p^+_k(dz-x)
\pr\left(\frac{z+S(n-k)}{c_{n-k}}\in A;\tau>n-k\right).
\end{align}
Choose now a sequence $\{N_n\}$ of integers satisfying 
\begin{equation}
\label{eq:thm1.2}
N_n=o(n)
\quad\text{and}\quad 
\frac{\delta_nc_n}{c_{N_n}}\to0.
\end{equation}
We start our analysis of the sum in \eqref{eq:thm1.1} by noting that 
\begin{align*}
&\sum_{k=N_n+1}^{n/2}\int_{(0,x_1]\times\R^{d-1}}p^+_k(dz-x)
\pr\left(\frac{z+S(n-k)}{c_{n-k}}\in A;\tau>n-k\right) \\
&\hspace{1cm}\le
\sum_{k=N_n+1}^{n/2}\pr(S_1(k)\ge -x_1;\tau^+>k)\pr(\tau>n-k)\\
&\hspace{1cm}\le\pr(\tau>n/2)
\sum_{k=N_n+1}^{n/2}\pr(S_1(k)\ge -x_1;\tau^+>k).
\end{align*}
Applying the second statement of Lemma~\ref{lem:lem20} to the walk 
$S_1(k)$ and recalling that the sequence
$\{c_k\}$ is regularly varying, we obtain 
\begin{align}
\label{eq:thm1.3}
\nonumber
\sum_{k=N_n+1}^{n/2}\pr(S_1(k)\ge -x_1;\tau^+>k)
&\le \sum_{k=N_n+1}^{\infty}\pr(S_1(k)\ge -x_1;\tau^+>k)\\
\nonumber
&\le \sum_{k=N_n+1}^{\infty} \frac{C_1x_1V(x_1)}{kc_k}
\le \frac{C_2x_1V(x_1)}{c_{N_n}}\\
&\le \frac{C_2\delta_n c_nV(x_1)}{c_{N_n}}.
\end{align}
taking into account \eqref{eq:thm1.2}, we conclude that 
\begin{equation}
\label{eq:thm1.4}
\sum_{k=N_n+1}^{\infty}\pr(S_1(k)\ge -x_1;\tau^+>k)=o(V(x_1))
\end{equation}
uniformly in $x_1\le \delta_nc_n$.
Consequently,
\begin{align}
\label{eq:thm1.5} 
\nonumber
&\sum_{k=N_n+1}^{n/2}\int_{(0,x_1]\times\R^{d-1}}p^+_k(dz-x)
\pr\left(\frac{z+S(n-k)}{c_{n-k}}\in A;\tau>n-k\right)\\
&\hspace{2cm}=o(V(x_1)\pr(\tau>n))
\end{align}
uniformly in $x_1\le \delta_nc_n$.

Using once again Lemma~\ref{lem:lem20}, we obtain
\begin{align*}
&\sum_{k=n/2}^{n}\int_{(0,x_1]\times\R^{d-1}}p^+_k(dz-x)
\pr\left(\frac{z+S(n-k)}{c_{n-k}}\in A;\tau>n-k\right)\\
&\hspace{1cm}\le
\sum_{k=n/2}^{n}\pr(S_1(k)\ge -x_1;\tau^+>k)\pr(\tau>n-k)\\
&\hspace{1cm}\le
\sum_{k=n/2}^{n}\frac{C_1x_1V(x_1)}{kc_k}\pr(\tau>n-k)
\le\frac{C_2x_1V(x_1)}{nc_n}\sum_{j=0}^n\pr(\tau>j).
\end{align*}
Since $\pr(\tau>j)$ is also regularly varying, we conclude that
\begin{align}
\label{eq:thm1.6}
\nonumber
&\sum_{k=n/2}^{n}\int_{(0,x_1]\times\R^{d-1}}p^+_k(dz-x)
\pr\left(\frac{z+S(n-k)}{c_{n-k}}\in A;\tau>n-k\right)\\
&\hspace{1cm}\le C\frac{x_1V(x_1)}{c_n}\pr(\tau>n)
=o(V(x_1)\pr(\tau>n))
\end{align}
uniformly in$x_1\le\delta_n c_n$.

Choose now a sequence $\varepsilon_n\to0$ so that 
$c_{N_n}=o(\varepsilon_n c_n)$.
By the convergence in the case of start at zero,
$$
\pr\left(\frac{z+S(n-k)}{c_{n-k}}\in A;\tau>n-k\right)
\sim \pr(M_{\alpha,\sigma}\in A)\pr(\tau>n)
$$
uniformly in $k\le N_n$ and $z\in B_{\varepsilon_n c_n}(0)$,
where $B_r(y)$ denotes the ball of radius $r$ with center at $y$.
Therefore,
\begin{align*}
 &\sum_{k=0}^{N_n}
 \int_{(0,x_1]\times\R^{d-1}\cap B_{\varepsilon_n c_n}(0)}p^+_k(dz-x)
\pr\left(\frac{z+S(n-k)}{c_{n-k}}\in A;\tau>n-k\right)\\
&\hspace{0.5cm}
=\left[\pr(M_{\alpha,\sigma}\in A)+o(1)\right]\pr(\tau>n)
\sum_{k=0}^{N_n}
\pr(S_1(k)\ge-x_1,|S(k)|<\varepsilon_nc_n;\tau^+>k)\\
&\hspace{0.5cm}
=\left[\pr(M_{\alpha,\sigma}\in A)+o(1)\right]\pr(\tau>n)
\sum_{k=0}^{N_n}
\pr(S_1(k)\ge-x_1;\tau^+>k)\\
&\hspace{1.5cm}
+O\left(\pr(\tau>n)\sum_{k=0}^{N_n}
\pr(S_1(k)\ge-x_1,|S(k)|\ge\varepsilon_nc_n;\tau^+>k)\right).
\end{align*}
By the definition of $V$,
$$
\sum_{k=0}^{N_n}\pr(S_1(k)\ge-x_1;\tau^+>k)=
V(x_1)-\sum_{k=N_n+1}^\infty
\pr(S_1(k)\ge-x_1;\tau^+>k).
$$
Using here \eqref{eq:thm1.4}, we obtain 
\begin{align}
\label{eq:thm1.7}
\nonumber
&\sum_{k=0}^{N_n}
 \int_{(0,x_1]\times\R^{d-1}\cap B_{\varepsilon_n c_n}(0)}p^+_k(dz-x)
\pr\left(\frac{z+S(n-k)}{c_{n-k}}\in A;\tau>n-k\right)\\
\nonumber
&\hspace{0.5cm}
=\left[\pr(M_{\alpha,\sigma}\in A)+o(1)\right]V(x_1)\pr(\tau>n)\\
&\hspace{1.5cm}
+O\left(\pr(\tau>n)\sum_{k=0}^{N_n}
\pr(S_1(k)\ge-x_1,|S(k)|\ge\varepsilon_nc_n;\tau^+>k)\right)
\end{align}
uniformly in $x_1\le\delta_nc_n$.

Furthermore,
\begin{align}
\label{eq:thm1.8} 
\nonumber
&\sum_{k=0}^{N_n}
 \int_{(0,x_1]\times\R^{d-1}\cap B^c_{\varepsilon_n c_n}(0)}p^+_k(dz-x)
\pr\left(\frac{z+S(n-k)}{c_{n-k}}\in A;\tau>n-k\right)\\
\nonumber
&\hspace{1cm}
\le \sum_{k=0}^{N_n} 
\pr(S_1(k)\ge-x_1,|S(k)|\ge\varepsilon_nc_n;\tau^+>k)\pr(\tau>n-k)\\
&\hspace{1cm}
\le C\pr(\tau>n)\sum_{k=0}^{N_n}
\pr(S_1(k)\ge-x_1,|S(k)|\ge\varepsilon_nc_n;\tau^+>k).
\end{align}
Having all these estimates one can easily see that it suffices to show that, uniformly in $x_1\le\delta_nc_n$,
\begin{equation}
\label{eq:thm1.9}
\sum_{k=0}^{N_n}
\pr(S_1(k)\ge-x_1,|S(k)|\ge\varepsilon_nc_n;\tau^+>k)=o(V(x_1)). 
\end{equation}
Indeed, applying \eqref{eq:thm1.9} to \eqref{eq:thm1.7} and \eqref{eq:thm1.8} leads us to the equality 
\begin{align*}
 &\sum_{k=0}^{N_n}\int_{(0,x_1]\times\R^{d-1}}p^+_k(dz-x)
\pr\left(\frac{z+S(n-k)}{c_{n-k}}\in A;\tau>n-k\right)\\
&\hspace{0.5cm}
=\left[\pr(M_{\alpha,\sigma}\in A)+o(1)\right]V(x_1)\pr(\tau>n).
\end{align*}
Plugging this and estimates \eqref{eq:thm1.5}, \eqref{eq:thm1.6}
into \eqref{eq:thm1.1}, we get 
$$
\pr\left(\frac{x+S(n)}{c_{n}}\in A;\tau_x>n\right)
=\left[\pr(M_{\alpha,\sigma}\in A)+o(1)\right]V(x_1)\pr(\tau>n)
$$
uniformly in $x_1\le \delta_nc_n$.  Recalling that 
$\pr(\tau_x>n)\sim V(x_1)\pr(\tau>n)$, we have uniform in 
$x_1\le\delta_nc_n$ convergence
$$
\pr\left(\frac{x+S(n)}{c_{n}}\in A;\tau_x>n\right)
\to \pr(M_{\alpha,\sigma}\in A).
$$

To prove \eqref{eq:thm1.9} we fix some $R\ge 1$ and notice that 
\begin{align*}
&\sum_{k=0}^{N_n}
\pr(S_1(k)\ge-x_1,|S(k)|\ge\varepsilon_nc_n;\tau^+>k)\\
&\hspace{1cm}\le \sum_{k=0}^{R}
\pr(|S(k)|\ge\varepsilon_nc_n;\tau^+>k)
+\sum_{k=R+1}^{N_n}
\pr(S_1(k)\ge-x_1;\tau^+>k).
\end{align*}
Similar to \eqref{eq:thm1.3},
$$
\sum_{k=R+1}^{N_n}\pr(S_1(k)\ge-x_1;\tau^+>k)
\le C\frac{x_1V(x_1)}{c_R}.
$$
Furthermore, using the convergence of measures $\pr(\frac{S(n)}{c_n}\in\cdot|\tau^+>n)$, we have 
$$
\sum_{k=0}^{R}\pr(|S(k)|\ge\varepsilon_nc_n;\tau^+>k)
=o\left(\sum_{k=0}^{R}\pr(\tau^+>k)\right)
=o(R\pr(\tau^+>R))
$$
uniformly in $R\le N_n$. Fix some $\gamma>0$ and take $R$ such that $c_{R-1}<x_1/\gamma$ and $c_R\ge x_1/\gamma$. Then 
$$
\sum_{k=R+1}^{N_n}\pr(S_1(k)\ge-x_1;\tau^+>k)
\le C\gamma V(x)
$$
and, by Lemma~\ref{Renew2},
$$
R\pr(\tau^+>R)\le CV(c_R)\le C\gamma^{-\alpha(1-\rho)}V(x_1).
$$
Combining these estimates, we conclude that 
$$
\lim_{n\to\infty}\sup_{x_1\le\delta_nc_n}\frac{1}{V(x_1)}
\sum_{k=0}^{N_n}
\pr(S_1(k)\ge-x_1,|S(k)|\ge\varepsilon_nc_n;\tau^+>k)
\le C\gamma.
$$
Since $\gamma$ can be chosen arbitrary small we get \eqref{eq:thm1.9}. Thus, the proof of the theorem is complete.
\end{proof}
\begin{proof}[First proof of Theorem~\ref{NormalDev'}]
\textcolor{blue}{
The give a proof in the non-lattice case only. The lattice case is even simpler.}

As in the proof of Theorem~\ref{Cru}, it suffices to consider the case $x=(x_1,0,\ldots,0)$ with $x_1\in(0,\delta_nc_n]$. The case $x_1=0$ has been considered in Section~\ref{sec:normal}. There we have proven that, uniformly in $y\in\mathbb{H}^+$,
\begin{equation}
\label{eq:thm2.0}
b_n(y)\sim \frac{\pr(\tau>n)}{c_n^d}p_{M_{\alpha,\sigma}}(y/c_n).
\end{equation}

To generalize this relation to the case of positive $x_1$, we first notice that, by Lemma~\ref{lem.arbitrary.starting point},
\begin{align}
\label{eq:thm2.1}
p_n(x,y)=\sum_{k=0}^n\int_{(0,x_1]\times\R^{d-1}}p^+_k(dz-x)
b_{n-k}(y-z).
\end{align}
Fix some $\gamma\in(0,1/2)$. The analysis of 
$$
\sum_{k=0}^{(1-\gamma)n}\int_{(0,x_1]\times\R^{d-1}}p^+_k(dz-x)
b_{n-k}(y-z)
$$
is very similar to our proof of Theorem~\ref{Cru}. If $k\le(1-\gamma)n$ then we have the bound
$$
b_{n-k}(y-z)\le C(\gamma)\frac{\pr(\tau>n)}{c_n^d},
$$
which is an immediate consequence of \eqref{eq:thm2.0}. Using this uniform bound and the local limit theorem \eqref{eq:thm2.0}
directly, and repeating our arguments from the proof of Theorem~\ref{Cru}, one can easily obtain 
\begin{equation}
\label{eq:thm2.2}
\sup_y\left|\frac{c_n^d}{\pr(\tau>n)}\sum_{k=0}^{(1-\gamma)n}\int_{(0,x_1]\times\R^{d-1}}p^+_k(dz-x)
b_{n-k}(y-z)-p_{M_{\alpha,\sigma}}(y/c_n)\right|\to0
\end{equation}
uniformly in $x_1\le\delta_n c_n$.

For $k>(1-\gamma)n$ the mentioned above bound for $b_{n-k}$
is useless, and one needs an additional argument.  We first notice that 
\begin{align*}
&\int_{(0,x_1]\times\R^{d-1}}p^+_k(dz-x)b_{n-k}(y-z)\\
&\hspace{1cm}
\le\sum_{u\in [0,x_1]\times\R^{d-1}\cap\Z^d}p_k^+(u-x+\Delta)
\max_{z\in u+\Delta}b_{n-k}(y-z)\\
&\hspace{1cm}
\le\sum_{u\in [0,x_1]\times\R^{d-1}\cap\Z^d}p_k^+(u-x+\Delta)
\pr(S(n-k)\in y-u-\mathbf{1}+2\Delta;\tau>n-k).
\end{align*}
Applying now the second statement of Lemma~\ref{lem:lem20}, we get 
\begin{align*}
&\int_{(0,x_1]\times\R^{d-1}}p^+_k(dz-x)b_{n-k}(y-z)\\
&\hspace{1cm}\le C\frac{V(x_1)}{kc_k^d}
\sum_{u\in [0,x_1]\times\R^{d-1}\cap\Z^d}
\pr(S(n-k)\in y-u-\mathbf{1}+2\Delta;\tau>n-k)\\
&\hspace{1cm}\le C\frac{V(x_1)}{kc_k^d}\pr(\tau>n-k).
\end{align*}
Consequently,
\begin{align*}
\sum_{k=(1-\gamma)n}^n
\int_{(0,x_1]\times\R^{d-1}}p^+_k(dz-x)b_{n-k}(y-z)
\le C\frac{V(x_1)}{nc_n^d}
\sum_{j=0}^{\gamma n}\pr(\tau>j).
\end{align*}
Noting now that the regular variation of $\pr(\tau>j)$
implies 
$$
\sum_{j=0}^{\gamma n}\pr(\tau>j)
\le C\gamma^{1-\rho}n\pr(\tau>n),
$$
we conclude that
$$
\lim_{n\to\infty}\frac{c_n^d}{\pr(\tau>n)}
\sum_{k=(1-\gamma)n}^n
\int_{(0,x_1]\times\R^{d-1}}p^+_k(dz-x)b_{n-k}(y-z)
\le C\gamma^{1-\rho}V(x_1)
$$
for all $x_1\le \delta_n c_n.$ Plugging this and \eqref{eq:thm2.2} into \eqref{eq:thm2.1} and letting $\gamma\to0$, we get the desired result.
\end{proof}

\begin{proof}[Second proof of Theorem~\ref{NormalDev'}]
    If the local assumption \eqref{eq:local.heavy} holds then we can use an approach similar to that 
    of in~\cite{DW15}, see Theorems 5 and 6 there.  This approach allows one to avoid considering first the special case $x=0$, as it was done in Section~\ref{sec:normal}. Without loss of generality we may assume that $x=(x_1,0,\ldots,0)$.  
    
    We first notice that if $y$ is such that $y_1\le 2\varepsilon c_n$ then, combining Lemma~\ref{lem:lem20} and Lemma~\ref{Renew2},
\begin{equation*}
p_n(x,y)\le C\frac{V(x_1)U(2\varepsilon c_n)}{nc_n^d}
\le CV(x_1)\pr(\tau_x>n)\varepsilon^{\alpha\rho}\frac{1}{c_n^d}. 
\end{equation*}
Thus, uniformly in $y\in\mathbb{H}^+$ with $y\le 2\varepsilon c_n$,
$$
c_n^d\pr(x+S(n)\in y+\Delta|\tau_x>n)\le C\varepsilon^{\alpha\rho}.
$$
Combining this bound with the fact that $p_{M_{\alpha,\sigma}}(z)$
to $0$ as $z\to\partial\mathbb{H}^+$, we conclude that 
\begin{align}
\label{eq:llt_small}
\lim_{\varepsilon\to0} 
\sup_{y\in \mathbb{H}^+:\,y_1\le2\varepsilon c_n}\left|c_{n}^d\pr\left(x+S(n)\in y+\Delta \mid \tau_x>n\right)- {p_{M_{\alpha ,\sigma}}\left(\frac{y-x}{c_{n}}\right)}\right|=0
\end{align}
uniformly in $x$ with $x_1\le\delta_nc_n.$

We next consider large values of $y$. More precisely, we assume that $|y|>3A c_n$ with some $A>1$. In this case we have, by the Markov property at time $m=[n/2]$, 
\begin{align*}
p_n(x,y)&=\int_{I(x)}\pr(x+S(m)\in dz,\tau_x>m)p_{n-m}(z,y)\\
&\hspace{2cm}+\int_{\mathbb{H}^+\setminus I(x)}\pr(x+S(m)\in dz,\tau_x>m)p_{n-m}(z,y),
\end{align*}
where $I(x)=\{z: |z-x|\le Ac_n\}$. 
If $z\in I(x)$ then $|y-z|>Ac_n$ for all sufficiently large values of $n$. Using \eqref{eq:local.ld}, we have 
$$
p_{n-m}(z,y)\le C\frac{n\phi(Ac_n)}{c_n^d}\le CA^{-\alpha}\frac{1}{c_n^d}.
$$
Therefore,
\begin{align*}
c_n^d\int_{I(x)}\pr(x+S(m)\in dz,\tau_x>m)p_{n-m}(z,y) 
\le CA^{-\alpha}\pr(\tau_x>m).
\end{align*}
Furthermore, using the standard concentration function estimate, we have 
\begin{align*}
&c_n^d \int_{\mathbb{H}^+\setminus I(x)}\pr(x+S(n)\in dz,\tau_x>m)p_{n-m}(z,y)\\
&\le C\pr(x+S(m)\notin I(x);\tau_x>m)
\le C\pr(|S(m)|>Ac_n;\tau_x>m).
\end{align*}
Combining these bounds, one gets easily 
\begin{align*}
c_n^d\pr(x+S(n)\in y+\delta|\tau_x>n)
\le C\left(A^{-\alpha}+\pr(|S(m)|>Ac_n|\tau_x>m)\right).
\end{align*}
Applying now the integral limit theorem, we conclude that 
\begin{align}
\label{eq:llt_large}
\lim_{A\to\infty}\limsup_{n\to\infty}
\sup_{y\in \mathbb{H}^+:\,|y|> 3Ac_n}\left|c_{n}^d\pr\left(x+S(n)\in y+\Delta \mid \tau_x>n\right)- {p_{M_{\alpha ,\sigma}}\left(\frac{y-x}{c_{n}}\right)}\right|=0
\end{align}
uniformly in $x$ with $x_1\le \delta_nc_n.$

Thus, it remains to consider $y$ such that $y_1>2\varepsilon c_n$ and $|y|\le 3Ac_n$. To analyse this range of values of $y$ we set 
$m=[(1-\gamma)n]$ with some $\gamma<1/2$. Let  $B_{\varepsilon c_n}(y)$ denote the ball of radius $\varepsilon c_n$ around $y$.
Then, by the Markov property at time $m$, we have 
\begin{align}
\label{eq:llt_normal.1}
\nonumber
p_n(x,y)&=\int_{B_{\varepsilon c_n}(y)}\pr(x+S(m)\in dz,\tau_x>m)p_{n-m}(z,y)\\
&\hspace{2cm}+
\int_{\mathbb{H}^+\setminus B_{\varepsilon c_n}(y)}\pr(x+S(m)\in dz,\tau_x>m)p_{n-m}(z,y).
\end{align}
Using the large deviations bound \eqref{eq:local.ld.2}, one gets easily
$$
\sup_{z\in \mathbb{H}^+\setminus B_{\varepsilon c_n}(y)}
p_{n-m}(z,y)
\le C\frac{(n-m)\phi(\varepsilon c_n)}{(\varepsilon c_n)^d}
\le C \gamma\varepsilon^{-d-\alpha}c_n^{-d}.
$$
Consequently,
\begin{equation}
\label{eq:llt_normal.2}
c_n^d \int_{\mathbb{H}^+\setminus B_{\varepsilon c_n}(y)}\pr(x+S(m)\in dz,\tau_x>m)p_{n-m}(z,y)
\le C\pr(\tau_x>n)\gamma\varepsilon^{-d-\alpha}.
\end{equation}
For the integral over $B_{\varepsilon c_n}(y)$ we have 
\begin{align*}
&\int_{B_{\varepsilon c_n}(y)}
\pr(x+S(m)\in dz,\tau_x>m)p_{n-m}(z,y) \\
&\hspace{0.2cm}=\int_{B_{\varepsilon c_n}(y)}
\pr(x+S(m)\in dz,\tau_x>m)\pr(z+S(n-m)\in y+\Delta)\\
&\hspace{0.5cm}+\int_{B_{\varepsilon c_n}(y)}
\pr(x+S(m)\in dz,\tau_x>m)\pr(z+S(n-m)\in y+\Delta;\tau_x\le n-m).
\end{align*}
By the strong Markov property, 
\begin{align*}
&\pr(z+S(n-m)\in y+\Delta;\tau_x\le n-m)\\
&=\sum_{k=1}^{n-m-1}\int_{\mathbb{H}^-}\pr(z+S(n)\in du,\tau_z=k)
\pr(u+S(n-m-k)\in y+\Delta).
\end{align*}
Noting that $|y-u|>\varepsilon c_n$ for all $u\in\mathbb{H}^-$ and using once again \eqref{eq:local.ld.2}, we obtain
\begin{equation*}
 c_n^d\pr(z+S(n-m)\in y+\Delta;\tau_x\le n-m)
 \le C\gamma\varepsilon^{-d-\alpha}.
\end{equation*}
Consequently,
\begin{align}
\label{eq:llt_normal.3}
\nonumber
&c_n^d\int_{B_{\varepsilon c_n}(y)}\pr(x+S(m)\in dz,\tau_x>m)
\pr(z+S(n-m)\in y+\Delta;\tau_x\le n-m)\\
&\hspace{2cm}\le C\gamma\varepsilon^{-d-\alpha}\pr(\tau_x>m).
\end{align}
By the Stone local limit theorem,
\begin{align*}
&\int_{B_{\varepsilon c_n}(y)}
\pr(x+S(m)\in dz,\tau_x>m)\pr(z+S(n-m)\in y+\Delta)\\
&\hspace{1cm}
=\frac{1+o(1)}{c_{n-m}^d}\int_{B_{\varepsilon c_n}(y)}
\pr(x+S(m)\in dz,\tau_x>m)
g_{\alpha,\sigma}\left(\frac{y-z}{c_{n-m}}\right).
\end{align*}
Recalling that $g_{\alpha,\sigma}$ is bounded, we then get
\begin{align}
\label{eq:llt_normal.4}
&\frac{c_n^d}{\pr(\tau_x>m)}\int_{B_{\varepsilon c_n}(y)}
\pr(x+S(m)\in dz,\tau_x>m)\pr(z+S(n-m)\in y+\Delta)\\
\nonumber
&\hspace{0.5cm}
=\gamma^{-d/\alpha}\e\left[
g_{\alpha,\sigma}\left(\frac{y-x-S(m)}{c_{n-m}}\right)
{\rm 1}\left\{x+S(m)\in B_{\varepsilon c_n}(y)\right\}\Big|\tau_x>m\right]+o(1).
\end{align}
Applying now the integral limit theorem, we infer that 
\begin{align*}
&\e\left[
g_{\alpha,\sigma}\left(\frac{y-x-S(m)}{c_{n-m}}\right)
{\rm 1}\left\{x+S(m)\in B_{\varepsilon c_n}(y)\right\}\Big|\tau_x>m\right]\\
&=\e\left[
g_{\alpha,\sigma}\left(\frac{y-x}{\gamma^{1/\alpha}c_{n}}
-\frac{(1-\gamma)^{1/\alpha}}{\gamma^{1/\alpha}}
M_{\alpha,\sigma}\right)
{\rm 1}\left\{M_{\alpha,\sigma}\in 
B_{\varepsilon(1-\gamma)^{-1/\alpha}}\left(\frac{y-x}{c_n}\right)\right\}\right]+o(1)\\
&=\e\left[
g_{\alpha,\sigma}\left(\frac{y-x}{\gamma^{1/\alpha}c_{n}}
-\frac{(1-\gamma)^{1/\alpha}}{\gamma^{1/\alpha}}
M_{\alpha,\sigma}\right)\right]\\
&-\e\left[
g_{\alpha,\sigma}\left(\frac{y-x}{\gamma^{1/\alpha}c_{n}}
-\frac{(1-\gamma)^{1/\alpha}}{\gamma^{1/\alpha}}
M_{\alpha,\sigma}\right)
{\rm 1}\left\{M_{\alpha,\sigma}\notin 
B_{\varepsilon(1-\gamma)^{-1/\alpha}}\left(\frac{y-x}{c_n}\right)\right\}\right]+o(1).
\end{align*}
Since $g_{\alpha,\sigma}(y)\le C|y|^{-d-\alpha}$,
\begin{align}
\label{eq:llt_normal.5}
\nonumber
&\e\left[
g_{\alpha,\sigma}\left(\frac{y-x}{\gamma^{1/\alpha}c_{n}}
-\frac{(1-\gamma)^{1/\alpha}}{\gamma^{1/\alpha}}
M_{\alpha,\sigma}\right)
{\rm 1}\left\{M_{\alpha,\sigma}\notin 
B_{\varepsilon(1-\gamma)^{-1/\alpha}}\left(\frac{y-x}{c_n}\right)\right\}\right]\\ 
&\le C\gamma^{(d+\alpha)/\alpha}\varepsilon^{-d-\alpha}.
\end{align}
Finally, we notice that, uniformly in $w\in\mathbb{H}^+$,
\begin{align}
\label{eq:llt_normal.6}
\nonumber
&\e\left[
g_{\alpha,\sigma}\left(\frac{w}{\gamma^{1/\alpha}}
-\frac{(1-\gamma)^{1/\alpha}}{\gamma^{1/\alpha}}
M_{\alpha,\sigma}\right)\right]\\
\nonumber
&\hspace{1cm}=\int_{\R^d}g_{\alpha,\sigma}\left(\frac{w}{\gamma^{1/\alpha}}
-\frac{(1-\gamma)^{1/\alpha}}{\gamma^{1/\alpha}}u\right)
p_{M_{\alpha,\sigma}}(u)du\\
\nonumber
&\hspace{1cm}=\gamma^{d/\alpha}\int_{\R^d}g_{\alpha,\sigma}\left(v\right)
p_{M_{\alpha,\sigma}}\left(\frac{w-\gamma^{1/\alpha}v}{(1-\gamma)^{1/\alpha}}\right)dv\\
&\hspace{1cm}
=\gamma^{d/\alpha}(p_{M_{\alpha,\sigma}}(w)+o(1))
\quad\text{as }\gamma\to0.
\end{align}
Combining \eqref{eq:llt_normal.3}--\eqref{eq:llt_normal.6},
we conclude that
\begin{align*}
\lim_{\gamma\to0}\limsup_{n\to\infty}\Bigg|&\frac{c_n^d}{\pr(\tau_x>m)} \int_{B_{\varepsilon c_n}(y)}
\pr(x+S(m)\in dz,\tau_x>m)p_{n-m}(z,y)\\
&\hspace{2cm}-p_{M_{\alpha,\sigma}}\left(\frac{y-x}{c_n}\right)\Bigg|=0.
\end{align*}
From this relation and from \eqref{eq:llt_normal.2} we get 
\begin{align*}
\sup_{y\in \mathbb{H}^+:\,y_1>2\varepsilon c_n|y|\le 3Ac_n}\left|c_{n}^d\pr\left(x+S(n)\in y+\Delta \mid \tau_x>n\right)- {p_{M_{\alpha ,\sigma}}\left(\frac{y-x}{c_{n}}\right)}\right|\to0 
\end{align*}
uniformly in $x$ with $x_1\le\delta_n c_n$.
This completes the proof of the theorem. 
\end{proof}


\section{Probabilities of small deviations when random walk starts at the origin}\label{sec:small}

\begin{proposition}
    \label{SmallDev'} Suppose $X\in \mathcal{D}(d,\alpha ,\sigma)$ and the
    distribution of $X$ is non-lattice. Then
    \begin{equation}
    c_{n}^d\pr(S(n)\in y+\Delta \mid \tau>n)\sim 
    g_{\alpha,\sigma}\left(
        0, \frac{y_{(2,d)}}{c_n}
    \right)
    \frac{\int_{y_{1}}^{y_{1}+\Delta }H(u)du}{n\pr\left( \tau
    ^{-}>n\right) },\quad \text{as } n\rightarrow \infty  \label{SD'}
    \end{equation}%
    uniformly in $y_{1}\in (0,\delta _{n}c_{n}]$, where $\delta _{n}\rightarrow 0$
    as $n\rightarrow \infty $.
\end{proposition}

\begin{proof} 
First, using once again \eqref{eq:int_A}, we get 
\begin{align*}
    {R}_{\varepsilon }^{(0)}(y)+{R}_{\varepsilon }^{(1)}(y)+{R}_{\varepsilon }^{(2)}(y)
     \le    
    \frac{CH(y_1)}{nc_n^d}
    \sum_{k=1}^{(1-\varepsilon) n}
     \pr(S_{1}(k)\in(0,y_1+2)).   
\end{align*}
When  $\alpha\in(1,2]$, using the Stone theorem, we proceed as follows 
\begin{align*}
    \sum_{k=1}^{(1-\varepsilon) n}
     \pr(S_{1}(k)\in (0,y_1+2))    
     &\le \sum_{k=1}^{(1-\varepsilon) n} \frac{C(y_1+3)}{c_k}\\      
     &\le C (y_1+3) \frac{n}{c_n}\le C\delta_n n.
\end{align*}    
Now we will consider the case $\alpha\le 1$.
Fix $\beta>0$ and notice that $1/\alpha-1+\beta>0$ for any $\alpha\le1$. 
Since  $c_n$ is regularly varying of index $1/\alpha$, 
by Potter's bounds,   there exists $C>0$, such that for $k\le n$, 
\[
\frac{c_{n}}{c_k} \le C \left(\frac{n}{k}\right)^{1/\alpha+\beta}. 
\]
Then, for the sequence $\gamma_n=\delta_n^{\frac{1}{2(1/\alpha-1+\beta)}}
\to 0$, 
\begin{align*}
    &\sum_{k=1}^{(1-\varepsilon) n}
     \pr(S_{1}(k)\in (0,y_1+2))   \\ 
     &\hspace{1cm}\le \gamma_n n + 
     \sum_{k=\gamma_n n}^{(1-\varepsilon) n}
     \pr(S_{1}(k)\in (0,y_1+2))\\
    &\hspace{1cm}\le 
    \gamma_n n +
    C\sum_{k=\gamma_n n}^{(1-\varepsilon) n} \frac{y_1+2}{c_k}
    \le 
    \gamma_n n + C \frac{y_1+2} {c_n} 
    \sum_{k=\gamma_n n}^{(1-\varepsilon) n} \left(\frac{n}{k}\right)^{1/\alpha+\beta}\\
    &\hspace{1cm}\le \gamma_n n + C \frac{y_1+2} {c_n} n \left(\frac{n}{\gamma_n n}\right)^{1/\alpha+\beta-1}
    \le \gamma_n n +Cn\delta_n\left(\frac{1}{\gamma_n }\right)^{1/\alpha+\beta-1}\\
    &\hspace{1cm}\le Cn(\gamma_n+\delta_n^{1/2}).
\end{align*}
Therefore,
\[
    {R}_{\varepsilon }^{(0)}(y)+{R}_{\varepsilon }^{(1)}(y)+{R}_{\varepsilon }^{(2)}(y)
    \le C
    (\gamma_n+\delta_n^{1/2}) \frac{H(y_1)}{c_n^d}, 
\]
and, as a result, for any fixed $\varepsilon>0$, 
\begin{equation}
    \lim_{n\rightarrow \infty }\sup_{y\in\R^d: 0<y_1\leq \delta _{n}c_{n}}\frac{c_{n}^d}{H(y_1)}%
    \left({R}_{\varepsilon }^{(0)}(y)+ R_{\varepsilon }^{(1)}(y)+R_{\varepsilon }^{(2)}(y)\right) =0.
    \label{Y1}
\end{equation}%
Since~\eqref{Y1} holds for any fixed $\varepsilon>0$ there exists a sequence $\varepsilon_n\downarrow 0$, such that~\eqref{Y1} is still true. Moreover, it will be true for any  $\varepsilon'_n\downarrow 0$ such that $\varepsilon'_n\ge \varepsilon_n$.
We will assume now that $\varepsilon = \varepsilon_n$.
It is clear and will be used in the subsequent proof that we can increase $\varepsilon_n$ and~\eqref{Y1}  will hold as long as $\varepsilon_n<1/2$. 

Now we represent 
\[
    R_\varepsilon^{(3)}(y)= 
R_\varepsilon^{(4)}(y)+R_\varepsilon^{(5)}(y)+R_\varepsilon^{(6)}(y), 
\]
where 
\begin{align*}
    R_\varepsilon^{(4)}(y)&=\pr({S}(n)\in y+\Delta)+\sum_{k=1}^{\varepsilon n}
    \int_{z_1\in (0,y_1), |z^{(2,d)}|\le  \varepsilon^{1/3} c_n}
    \pr(S(n-k)\in  y-z+\Delta)dB_{k}(z)\\
R_\varepsilon^{(5)}(y)&=\sum_{k=1}^{\varepsilon n}
    \int_{z_1\in (y_1,y_1+1), |z^{(2,d)}|\le  \varepsilon^{1/3} c_n}
    \pr(S(n-k)\in  y-z+\Delta)dB_{k}(z)\\
    R_\varepsilon^{(6)}(y)&=\sum_{k=1}^{\varepsilon n}
\int_{z\in \R^d: |z^{(2,d)}|> \varepsilon^{1/3} c_n}
\pr(S(n-k)\in  y-z+\Delta, S(n-k)_{1}>0)dB_{k}(z).
\end{align*}
Let $\varepsilon_n$ an arbitrary converging to zero sequence of positive numbers.
First, by the Stone local limit theorem, 
\begin{multline*}
\int_{0<z_1<y_1, |z|^{(2,d)}\le \varepsilon_n^{1/3} c_n}\pr(S(n-k)\in  y-z+\Delta)dB_{k}(z)\\
=\frac{g_{\alpha,\sigma }
\left(0,\frac{y_{(2,d)}}{c_{n-k}}\right)+\Delta _{1}(n-k,y)}{c_{n-k}^d}
\int_{0<z_1<y_1, |z^{(2,d)}|\le \varepsilon_n^{1/3} c_n}dB_{k}(z),
\end{multline*}%
where $\Delta _{1}(n-k,y)\rightarrow 0$ uniformly in $z$ such that
$z_1\in (0,\delta_{n}c_{n})$ and $k\in \lbrack 1, \varepsilon_n n]$. 
Therefore,
\begin{align}\nonumber
    R_{\varepsilon_n}^{(4)}(y)
    &= 
    \left(g_{\alpha,\sigma}
\left(0,\frac{y_{(2,d)}}{c_{n}}\right)+\Delta _{1}(n,y)
    \right)
    \left(
        \frac{1}{c_n^d}
        +
        \sum_{k=1}^{[\varepsilon_n n]}
        \frac{1}{c_{n-k}^d}
        \int_{0<z_1<y_1, |z^{(2,d)}|\le \varepsilon_n^{1/3} c_n}dB_{k}(z)
    \right)\\
    \label{eq:r4.1}
    &= 
    \frac{g_{\alpha,\sigma}
\left(0,\frac{y_{(2,d)}}{c_{n}}\right)+\Delta _{1}(n,y)
    }
    {c_n^d}
    \left(
        1
        +
        \sum_{k=1}^{[\varepsilon_n n]}
        \int_{0<z_1<y_1, |z^{(2,d)}|\le \varepsilon_n^{1/3} c_n}dB_{k}(z),
    \right) 
\end{align}
where $\Delta _{1}(n,y)\rightarrow 0$ uniformly in $y$ with
$y_1\in (0,\delta_{n}c_{n}) $ and $k\in \lbrack 1, \varepsilon_n n]$.
Now we represent
\begin{multline}\label{eq:r4.2}
1+\sum_{k=1}^{\varepsilon_n n}
        \int_{0<z_1<y_1, |z^{(2,d)}|\le \varepsilon_n^{1/3} c_n}dB_{k}(z)
  =H(x_1)-\sum_{k=\varepsilon_n n}^\infty \pr(S_{1}(k),\tau>k)\\ -
\sum_{k=1}^{\varepsilon_n n}
        \int_{0<z_1<y_1, |z^{(2,d)}|>\varepsilon_n^{1/3} c_n}dB_{k}(z).
\end{multline}
For any fixed $\varepsilon>0$, by Corollary~22 of~\cite{VW2009},
\begin{align*}
    \sum_{k=\varepsilon n}^\infty \pr(S_{1}(k)<x_1,\tau>k) 
    &\le Cy_1H(y_1)
    \sum_{k>\varepsilon n}\frac{1}{kc_k}\\ 
    &\le C \frac{y_1H(y_1)}{c_{\varepsilon  n}}
    =o(H(y_1))
\end{align*}
uniformly in in $y$ with $y_1\le \delta_nc_n$. 
Hence, this bound  holds for some sequence $\varepsilon_n\downarrow 0$.
Increasing the original sequence $\varepsilon_n$ if needed 
we obtain 
\begin{equation}\label{eq:r4.3}
    \sum_{k=\varepsilon_n n}^\infty \pr(S_{1}(k)<y_1,\tau>k) 
    =o(H(y_1))
\end{equation}
uniformly in in $y$ with $y_1\le \delta_nc_n$. 

Fix a large positive number $A$ and let 
\[
    c^\leftarrow (y):=\inf\{\, k\ge 1: c_k>y \,\}.
\]
Note that $c_{c^\leftarrow (y)}>y$ and $c_{c^\leftarrow (y)-1}\le y$. 
Also note that since $\mu(y)$ is regularly varying, 
$c^\leftarrow (y)\sim 1/\mu(y)$ as  $y\to\infty$. 
Then,
\[
\sum_{k=1}^{\varepsilon_n n}
        \int_{0<z_1<y_1, |z^{(2,d)}|>\varepsilon_n^{1/3} c_n}dB_{k}(z)
\le R_{\varepsilon_n}^{(4,1)}(y)+
    R_{\varepsilon_n}^{(4,2)}(y),
    \]
where
\[
    R_{\varepsilon}^{(4,1)}(y)
    =\sum_{k=1}^{c^\leftarrow(Ay_1)-1}
    \pr(S_{1}(k)<y_1, |S_{2,d}(k)| >\varepsilon_n^{1/3} c_n, 
    \tau>k)
\]
and 
\[
    R_{\varepsilon}^{(4,2)}(y)
    =\sum_{k=c^\leftarrow(Ay_1)}^{\varepsilon_n n}
     \pr(S_{1}(k)<y_1, |S_{2,d}(k)| >\varepsilon_n^{1/3} c_n, 
    \tau>k).
\]
We can estimate  $R_{\varepsilon}^{(4,2)}(x)$ similarly to the above 
\begin{align*}
    R_{\varepsilon}^{(4,2)}(y)
    &\le 
    \sum_{k=c^\leftarrow(Ay_1)}^\infty
    \pr(S_1(k)<y_1,\tau>k) \\ 
    &\le C(y_1+1)H(y_1)
    \sum_{k=c^\leftarrow(Ay_1)}^\infty\frac{1}{kc_k}
    \le C \frac{(y_1+1)H(y_1)}{A(y_1+1)}=\frac{C}{A}H(y_1).
\end{align*}
Clearly, taking $A$ sufficiently large we can make this term much smaller than
$H(y_1)$.
By Theorem~\ref{Cru}, 
\[
    \pr(S_1(k)<y_1, |S_{2,d}(k)| >\varepsilon_n^{1/3} c_n\mid 
    \tau>k)
    \le    
    \pr(|S_{2,d}(k)| >\varepsilon_n^{1/3} c_n\mid 
    \tau>k)\to 0,
\] 
uniformly in $k\le \varepsilon_n n$. 
Combining this with \eqref{AsH}, we conclude that
\begin{align*}
    R_{\varepsilon_n}^{(4,1)}(y)
    &=o\left(
    \sum_{k=1}^{c^\leftarrow(Ay_1)-1} \pr(\tau>k)\right)\\
    &= o \left(\sum_{k=1}^{c^\leftarrow(Ay_1)-1} \frac{H(c_k)}{k}\right)=o\left(H(y_1)\right)
\end{align*}
uniformly in $y$ with $y_1\le \delta_n c_n$. 
Hence, 
\begin{equation}
    \label{eq:r4.4}
        \int_{0<z_1<y_1, |z^{(2,d)}|>\varepsilon_n^{1/3} c_n}dB_{k}(z)
        =o(H(y_1))
\end{equation}
uniformly in $y$ with $y_1\le \delta_nc_n$.
Combining~\eqref{eq:r4.1},\eqref{eq:r4.2},\eqref{eq:r4.3} and 
\eqref{eq:r4.4} we obtain that 
\begin{equation}\label{eq:r4}
    R_{\varepsilon_n}^{(4)}(y) =
    \frac{g_{\alpha ,\sigma }
\left(0,\frac{y_{(2,d)}}{c_{n}}\right)+\Delta _{2}(n,y)
    }{c_n^d} H(y_1), 
\end{equation}   
where $\Delta _{2}(n,y)\rightarrow 0$ uniformly in $y$ such that $y_1\in (0,\delta_{n}c_{n})$. 
Using~\eqref{eq:r4.4} and the Stone local limit theorem, one can easily conclude that 
 \begin{equation}\label{eq:r6}
    R_{\varepsilon_n }^{(6)}(y)
    \le \frac{C}{c_n^d}
    \sum_{k=1}^{\varepsilon_n n}
        \int_{0<z_1<y_1, |z^{(2,d)}|>\varepsilon_n^{1/3} c_n}dB_{k}(z)
        =o\left(\frac{H(y_1)}{c_n^d}\right).
 \end{equation}
 uniformly in $y$ such that $y_1\in (0,\delta _{n}c_{n})$. 

 Analysis of $R^{(5)}_\varepsilon(y)$ is very similar to that of $R^{(4)}_\varepsilon(y)$.
 First we make use of the Stone theorem,
 \begin{equation*}
    R_{\varepsilon }^{(5)}(y)
    = 
    \frac{g_{\alpha }
\left(0,\frac{y_{(2,d)}}{c_{n}}\right)+\Delta _{3}(n,y)}
    {c_n^d}
        \sum_{k=1}^{\varepsilon_n n}
        \int_{z_1\in (y_1,y_1+1) , |z^{(2,d)}|\le \varepsilon_n^{1/3} c_n}
        (z_1-y_1)
        dB_{k}(z),  
\end{equation*}
where $\Delta _{3}(n,y)\rightarrow 0$ uniformly in $y$ with 
$y_1\in (0,\delta_{n}c_{n})$. 
Then, using the same arguments as above, we obtain 
\[
    R_{\varepsilon }^{(5)}(y)
    = 
    \frac{g_{\alpha,\sigma }
\left(0,\frac{y^{(2,d)}}{c_{n}}\right)+\Delta _{4}(n,y)}
    {c_n^d}
        \sum_{k=1}^{\varepsilon_n n}
        \int_{z_1\in (y_1,y_1+1)}
        (1+z_1-y_1)
        dB_{k}(z),    
\]
where $\Delta _{4}(n,y)\rightarrow 0$ uniformly in 
$y$ with $y_1\in (0,\delta_{n}c_{n})$. 
Integrating by parts we can complete the proof now. 

\end{proof}


\section{Probabilities of small  deviations when random walks start at 
an arbitrary starting point}

\begin{proof}[Proof of Theorem~\ref{thm:small.deviations.arbitrary.lattice} in the lattice case]
    For the lattice distribution we have 
    from~\eqref{eq:g.arbitrary.starting.point} 
    \[
      p_n(x,y) = \pr(x+S(n)=y,\tau_x>n)  
      =\sum_{k=0}^n 
      \sum_{z_1=1}^{x_1\wedge y_1} 
      \sum_{z_2,\ldots,z_d}
      p_k^+(0,z-x)b_{n-k}(y-z). 
    \]
Let $N_n$ be the sequence of integers, which was constructed in the proof of Theorem~\ref{Cru}. 
We shall use the representation 
\[
    p_n(x,y) = P_1(x,y)+P_2(x,y)+P_3(x,y),
\]   
where 
\begin{align*}
    P_1(x,y)&:= \sum_{k=0}^{N_n} 
    \sum_{z_1=1}^{x_1\wedge y_1} 
    \sum_{z_2,\ldots,z_d}
    p_k^+(0,z-x)b_{n-k}(y-z),\\
    P_2(x,y)&:= \sum_{k=N_n+1}^{n-N_n-1} 
    \sum_{z_1=1}^{x_1\wedge y_1} 
    \sum_{z_2,\ldots,z_d}
    p_k^+(0,z-x)b_{n-k}(y-z),\\
    P_3(x,y)&:= \sum_{k=n-N_n}^n  
    \sum_{z_1=1}^{x_1\wedge y_1} 
    \sum_{z_2,\ldots,z_d}
    p_k^+(0,z-x)b_{n-k}(y-z).
\end{align*}    

To estimate $P_2(x,y)$ we shall proceed as in the analysis of normal deviations.
Note that  by Lemma~\ref{lem:lem20}, 
\[
b_{n-k}(y-z)
\le C \frac{H(y_1-z_1)}{(n-k)c_{n-k}^d}.
\]
Then
\begin{align*}
P_2(x,y)
&\le C\frac{H(y_1)}{N_nc_{N_n}^d}
\sum_{k=N_n+1}^\infty\sum_{z_1=1}^{x_1\wedge y_1} 
    \sum_{z_2,\ldots,z_d}p_k^+(0,z-x)\\
    &\le C\frac{H(y_1)}{N_nc_{N_n}^d}
\sum_{k=N_n+1}^\infty\pr(S_1(k)\ge -x_1;\tau^+>k).
\end{align*}   
Using now \eqref{eq:thm1.4} and increasing, if needed $N_n$, we conclude that 
\begin{equation}\label{eq:p2}
    P_2(x,y)
    = o\left(\frac{H(y_1) V(x_1)}{N_nc_{N_n}^d}\right)
    = o\left(\frac{H(y_1) V(x_1)}{nc_n^d}\right). 
\end{equation}     

Now we will consider the first term $P_1(x,y)$. 
Let $\varepsilon_n\downarrow 0 $ be the sequence, which we have defined in the proof of Theorem~\ref{Cru}.
We will need the following sets 
\begin{align*} 
    A_1(x,y)&=\left\{z: |x-z|\le \varepsilon_n c_n, z_1 \in  (0,x_1\wedge y_1 ]  \right\}
    \\
    C_1(x,y)&=\left\{z: |x-z|> \varepsilon_n c_n,  z_1 \in  (0,x_1\wedge y_1]  \right\}
    \end{align*}
 
 Applying now the asymptotics for small deviations of walks starting at zero, we get 
 \begin{align*} 
 &\sum_{k=0}^{N_n}
 \sum_{z\in A_1(x,y)}
 p_k^+(0,z-x)b_{n-k}(y-z)\\
 &\hspace{1cm}\sim  
\frac{g_{\alpha,\sigma} 
\left(0,\frac{y_{2,d}-x_{2,d}}{c_n}\right)}{nc_n^d}
\sum_{k=0}^{N_n}
 \sum_{z\in A_1(x,y)}
 p_k^+(0,z-x)H(y_1-z_1). 
 \end{align*} 
Next we note that 
\begin{align*}
&\sum_{k=0}^{N_n} \sum_{z\in C_1(x,y)}
    p_k^+(0,z-x)b_{n-k}(y-z)\\
    &\le C\frac{H(y_1)}{nc_n^d} 
    \sum_{k=0}^{N_n} 
    \pr(|S(k)|\ge \varepsilon_n c_n,S_1(k)>-x_1,\tau^+>k)
\end{align*}
and
\begin{align*}
&\sum_{k=0}^{N_n} \sum_{z\in C_1(x,y)}
    p_k^+(0,z-x)H(y_1-z_1)\\
    &\le H(y_1) 
    \sum_{k=0}^{N_n} 
    \pr(|S(k)|\ge \varepsilon_n c_n,S_1(k)>-x_1,\tau^+>k). 
\end{align*}    
Taking into account \eqref{eq:thm1.9}, we obtain 
\begin{equation}\label{eq:p1.1} 
P_1(x,y) \sim 
\frac{g_{\alpha,\sigma} 
\left(0,\frac{y_{2,d}-x_{2,d}}{c_n}\right)}{nc_n^d}
 \sum_{z_1=0}^{x_1\wedge y_1}
 (V(x_1-z_1)-V(x_1-z_1-1))H(y_1-z_1), 
\end{equation}      
where we also replaced $\sum_{0}^{N_n}$ 
by  $\sum_{0}^{\infty}$ using the 
arguments in the proof of Proposition~11 in~\cite{Doney2012}. 
Analogous arguments 
give us 
\[ 
P_3(x,y) \sim 
\frac{g_{\alpha,\sigma} 
\left(0,\frac{y_{2,d}-x_{2,d}}{c_n}\right)}{nc_n^d}
 \sum_{z_1=0}^{x_1\wedge y_1}
 V(x_1-z_1)(H(y_1-z_1)-H(y_1-z_1-1).
\] 
Then, the arguments at the end of the proof 
of Proposition~11 in~\cite{Doney2012} 
give 
\[
    P_1(x,y) +P_3(x,y) \sim  
    V(x_1)H(y_1)\frac{g\left(0,\frac{y_{2,d}-x_{2,d}}{c_n}\right)}{nc_n^d}
    .
\]

\end{proof}

\begin{proof}[Proof of Theorem~\ref{thm:small.deviations.arbitrary.lattice} in the non-lattice case]
    The proof is very similar to the proof in the  
    lattice case. 
    
    For the non-lattice distribution we will make use of~\eqref{eq:g.arbitrary.starting.point.interval} 
We split the sum as follows, 
\[
    p_n(x,y) = P_1(x,y)+P_2(x,y)+P_3(x,y),
\]   
where 
\begin{align*}
    P_1(x,y)&:= 
    \sum_{k=0}^{N_n} 
    \int_{(0,x_1\wedge (y_1+1) ] \times \R^{d-1}}
    p_k^+(0,dz-x)b_{n-k}(y-z),\\ 
    P_2(x,y)&:= \sum_{k=N_n+1}^{n-N_n-1} 
    \int_{(0,x_1\wedge (y_1+1) ] \times \R^{d-1}}
    p_k^+(0,dz-x)b_{n-k}(y-z),\\
    P_3(x,y)&:= \sum_{k=n-N_n}^n  
    \int_{(0,x_1\wedge (y_1+1) ] \times \R^{d-1}}
    p_k^+(0,dz-x)b_{n-k}(y-z).
\end{align*}    
There is virtually no difference in estimates 
for $P_2(x,y)$. So repeating the same arguments 
we obtain   
\begin{equation}\label{eq:p2.non-lattice}
    P_2(x,y) = o\left(\frac{H(y_1) V(x_1)}{nc_n^d}\right). 
\end{equation}     

Now we will consider the first term $P_1(x,y)$. 
We will need the following sets 
\begin{align*} 
    A_1(x,y)&=\left\{z: |x-z|\le \varepsilon_n c_n, z_1 \in  (0,x_1\wedge y_1 ]  \right\}
    \\
    C_1(x,y)&=\left\{z: |x-z|> \varepsilon_n c_n,  z_1 \in  (0,x_1\wedge y_1]  \right\}
    \end{align*}
 
 Now we have,
 \begin{align*} 
 &\sum_{k=0}^{N_n}
 \int_{A_1(x,y)}
 p_k^+(0,dz-x)b_{n-k}(y-z)\\
 &\hspace{1cm}\sim  
\frac{g_{\alpha,\sigma} 
\left(0,\frac{y_{2,d}-x_{2,d}}{nc_n^d}\right)}{nc_n^d}
\sum_{k=0}^{N_n}
 \int_{A_1(x,y)}
 p_k^+(0,dz-x)
 \int_{y_1-z_1}^{y_z-z_1+1}
 H(u)du.  
 \end{align*} 
Now note that by~\eqref{eq:thm1.9}
\begin{align*}
    \sum_{k=0}^{\varepsilon_n n} \int_{C_1(x,y)}
    p_k^+(0,dz-x)H(y_1-z_1)
    &\le H(y_1) 
    \sum_{k=0}^{\varepsilon_n n} 
    \pr(|S(k)|\ge \delta c_n,S_1(k)>-x_1,\tau^+>k)\\ 
    &=o(H(y_1)V(x_1)), 
\end{align*}    
provided that $\gamma_n$ and $\varepsilon_n$ 
converges to $0$ sufficiently slowly. 
As a result, 
\begin{equation}\label{eq:p1.1.non-lattice} 
P_1(x,y) \sim 
\frac{g_{\alpha,\sigma} 
\left(0,\frac{y_{2,d}-x_{2,d}}{c_n}\right)}{nc_n^d}
\int_{(0,x_1\wedge (y_1+1])}
 V(x_1-dz_1)\int_{y_1-z_1}^{y_z-z_1+1}
 H(u)du , 
\end{equation}      
where we also replaced $\sum_{0}^{\varepsilon_n n}$ 
by  $\sum_{0}^{\infty}$ using similar arguments. 
Analogous arguments 
give us 
\begin{equation}\label{eq:p3.1.non-lattice} 
    P_3(x,y) \sim 
    \frac{g_{\alpha,\sigma} 
    \left(0,\frac{y_{2,d}-x_{2,d}}{c_n}\right)}{nc_n^d}
    \int_{(0,x_1\wedge (y_1+1])}
     V(x_1-z_1)H(y_1-z_1+\Delta)dz_1. 
\end{equation}      
 Now note that integration by parts 
 of the first integral 
 gives 
 \begin{multline*} 
    \int_{(0,x_1\wedge (y_1+1])}
    V(x_1-dz_1)\int_{y_1-z_1}^{y_z-z_1+1}
    H(u)du
    +   
    \int_{(0,x_1\wedge (y_1+1])}
     V(x_1-z_1)H(y_1-z_1+\Delta)dz_1\\ 
     =V(x_1)\int_{y_1}^{y_1+1}H(u) du. 
 \end{multline*}   
As a result, 
\[
    P_1(x,y) +P_3(x,y) \sim  
    V(x_1)\int_{y_1}^{y_1+1}H(u) du
    \frac{g\left(0,\frac{y_{2,d}-x_{2,d}}{c_n}\right)}{nc_n^d}.
\]
\end{proof} 


\section{Probabilities of large deviations when random walk starts at the origin}\label{sec:large}

We will need the following large deviations estimates. 
\begin{proposition}\label{prop:local.ld}
    Let  $X\sim \mathcal D(d,\alpha,\sigma)$ with some $\alpha<2$.  
    Suppose that there exists a regularly varying $\varphi$ 
    such thatthe upper bound in~\eqref{eq:local.heavy.1} holds. 
    Then, there exists constant $C_H$ such that for $|x|\ge c_n$ we have
    \begin{equation}\label{eq:local.ld}
        \pr(S(n)\in x+\Delta) 
        \le C_H
        \frac{1}{c_n^d}
        n\phi(|x|).
    \end{equation}    
    
    If, in addition,~\eqref{eq:local.heavy}  holds, then 
    \begin{equation}\label{eq:local.ld.2}
        \pr(S(n)\in x+\Delta) 
        \le C_H n\frac{\phi(|x|)}{|x|^{d}}
    \end{equation}    
\end{proposition}
This result is proved in~\cite[Theorem~2.6]{Berger2019} in the lattice case. 
We omit the proof of non-lattice case, as it  can be done  very similarly to~\cite[Theorem~2.6]{Berger2019}. 

Using the definition~\eqref{Defa} of $c_n$ we obtain from Corollary~\ref{Cmin} 
the following upper bound. (Recall that 
$g(r)=\frac{\phi(|r|)}{r^d}$.) 
\begin{lemma}\label{lem:heavy.prelim}
    For any $A>1$ there exists $c_A$ such that 
    \begin{equation}\label{eq:heavy.prelim}
        b_n(x) \le c_A H(x_1) g(|x|),
    \end{equation}
    for $x$ with $c_n\le |x|\le Ac_n$. 
\end{lemma}
The main goal of this section is to obtain an upper bound for $b_n(x)$ in the case $|x|>Ac_n$. 
We now obtain a bound, which will be valid also for $|x|>Ac_n$.
 \begin{lemma}\label{lem:heavy.main.stable}
Suppose that $X$ is asymptotically stable with \(\alpha\in(0,2)\).     
If \eqref{eq:local.heavy} holds then there exists $\gamma>0$ 
    such that for all $y$ with $|y|>c_n$ we have 
    \begin{equation}\label{eq:heavy.main}
        b_n(y) \le \gamma H(y_1+1) g(|y|).
    \end{equation}    
\end{lemma}
\begin{proof}
We will first introduce some constants and sequences that will be used throughout the proof. 
Set
\[ 
\rho_n = \frac{1}{n}\sum_{k=1}^n \pr(S_1(k)>0).
\]
Fix $\delta\in(0,1)$ such that 
    \begin{equation}\label{eq:delta.stable}
        \mathrm{e}^{\delta}\sup_{n\ge1}\rho_n+\delta \mathrm{e}^{\delta}<1
    \end{equation}    
 and let \(\widetilde A\) be such that 
    \begin{multline}\label{eq:A1.tilde.stable}
    \int_{|z|> \widetilde A c_k, |y-z|>|y|/2}
    g(|y-z|) g(|z|) dz\\
    +\int_{|y-z|> \widetilde A c_k, |z|>|y|/2}
    g(|y-z|) g(|z|) dz
    \le\frac{\delta}{C_H k}g(|y|)
    \end{multline}
    for $y$ with $|y|>1$ and \(k\ge 1\). 
    Let $A$ be such that 
    \begin{equation}\label{eq:A1.stable}
\sup_{n\ge 1}\sup_{y,z: |y|>Ac_n, |z|\le \widetilde A c_n+1}\frac{g(|y-z|)}{g(|y|)}
\le \mathrm{e}^\delta. 
    \end{equation}
    By Lemma~\ref{lem:heavy.prelim} there exist $c_A>1$ 
    such that~\eqref{eq:heavy.main}  
    holds for $y$ with $c_n\le |y|\le Ac_n$. 
    
    The proof will be done by  induction. 
    We will inductively construct an increasing sequence \(\gamma_n\) such that 
    \begin{equation}\label{eq:ind.gamma}
        b_n(y) \le \gamma_n H(y_1) g(|y|)
    \end{equation}
    for $y$ with $|y|>c_n$ and \(n\ge 1\). 
    Then we will show that \(\sup_n \gamma_n <\infty\). 
    We put \(\gamma_1=c_A\) and then the base of induction \(n=1\) is immediate. 
    Since \(\gamma_n\) will be increasing,  
    it follows from the definition of \(A\) that~\eqref{eq:ind.gamma} 
    holds for 
    $y$ such that  $|c_n|<y\le A c_n$. 
    Hence, we will consider only $y$ with $|y|>c_n$.

    Assume now that we have already constructed the elements \(\gamma_k\) for \(k\le n-1\). We shall construct the next value $\gamma_n$. 
   It follows from~\eqref{eq:40} that 
    \begin{align}
        \nonumber 
    nb_n(y+\Delta)&=\pr(S(n)\in y+\Delta)\\ 
    &\hspace{1cm}+\sum_{k=1}^{n-1} 
    \int_{\R^d}\pr(S(k)\in y-z+\Delta, S_1(k)>0) dB_{n-k}(z)
    \nonumber\\
    &=: R^{(1)}(y)+R^{(2)}(y)+R^{(3)}(y),
    \label{c1.heavy.stale}
    \end{align}
where
\begin{align*}    
    R^{(1)}(y)&:=\pr(S(n)\in y+\Delta)
    +\sum_{k=1}^{n-1} \int_{|z|\le |y|/2}\pr(S(k)\in y-z+\Delta, S_1(k)>0) dB_{n-k}(z),\\
    R^{(2)}(y)&:=
    \sum_{k=1}^{n-1} \int_{|z|>|y|/2,|y-z|\le \widetilde A c_k}\pr(S(k)\in y-z+\Delta, S_1(k)>0) dB_{n-k}(z),\\
   R^{(3)}(y)&:=\sum_{k=1}^{n-1} \int_{|z|> |y|/2, |y-z|>\widetilde A c_k}\pr(S(k)\in y-z+\Delta, S_1(k)>0) dB_{n-k}(z).
\end{align*}
Using first the local large deviations bound~\eqref{eq:local.ld} and 
then the regular variation of $g$, we get
\begin{align}
    \nonumber
    R^{(1)}(y)&\le C_H ng(|y|)  +
    C_H\sum_{k=1}^{n-1} \int_{|z|\le |y|/2, 0\le z_1\le y_1+1}
     k g(|y-z|) dB_{n-k}(z)\\
    \nonumber&\le C_H ng(|y|)  +
    C_g C_H g(|y|)\sum_{k=1}^{n-1} k \int_{|z|\le |y|/2, 0\le z_1\le y_1+1} dB_{n-k}(z)\\
    \label{eq:R1.stable}&\le (C_g+1)C_H ng(|y|)H(y_1+1).
\end{align} 
Second, integrating by parts and using then the definition~\eqref{eq:A1.stable} of $A$
 and  the induction assumption we obtain 
\begin{align} 
    \nonumber
    R^{(2)}(y)
    &\le
    \sum_{k=1}^{n-1} \int_{|y|>|z|/2,|y-z|\le \widetilde A c_k,
    0\le z_1\le y_1+1}
    \pr(S(k)\in y-z+\Delta, S_1(k)>0) dB_{n-k}(z)
    \\    
    \nonumber 
    &\le
    \sum_{k=1}^{n-1} \int_{|y-z|>|z|/2-1,|z|\le \widetilde A c_k+1,
    0\le z_1\le y_1+1}
    b_{n-k}(y-z)\pr(S(k)\in dz, S_1(k)>0) 
    \\    
    \nonumber 
    &\le
    \sum_{k=1}^{n-1} \int_{|y-z|>|z|/2-1,|z|\le \widetilde A c_k+1,
    0\le z_1\le y_1+1}
    \gamma_{n-1} g(|y-z|) H(y_1-z_1+1)\pr(S(k)\in dz, S_1(k)>0) 
    \\    
    &\nonumber 
   \le \mathrm{e}^\delta \gamma_{n-1} g(|y|)H(y_1+1)
   \sum_{k=1}^{n-1} \pr(S_1(k)>0)\\
   &= \mathrm{e}^{\delta} (n-1)  \gamma_{n-1} \rho_{n-1}
   g(|y|)H(y_1+1).\label{eq:R2.stable}
\end{align} 
Third, using the induction assumption and~\eqref{eq:local.ld.2},  
\begin{align*}
    R^{(3)}(y)
    &\le
    C_H\gamma_{n-1}\sum_{k=1}^{n-1} k \int_{|z|> |y|/2, |x-y|>\widetilde A c_k,0\le z_1\le y_1+1}
    g(|y-z|) g(|z|) H(z_1)dz \\
    &\le
    C_H\gamma_{n-1}H(y_1+1) \sum_{k=1}^{n-1} k 
    \int_{|z|> |y|/2, |y-z|>\widetilde A c_k,
    0\le z_1\le y_1+1}
    g(|y-z|) g(|z|) dz.
\end{align*}
We can estimate the integral as follows, 
\[
\int_{|y-z|> \widetilde A c_k, |z|>|y|/2}
    g(|y-z|) g(|z|) dz
   \le\frac{\delta}{C_H k}g(|y|), 
\] 
using the definition~\eqref{eq:A1.tilde.stable} of 
 $\widetilde A$. 
Hence, 
\begin{equation}\label{eq:R3.stable}
    R^{(3)}(y)\le 
    C_H\gamma_{n-1}H(y_1+1) 
    \frac{\delta}{C_H }g(|y|)(n-1)
    \le \delta \gamma_{n-1}\mathrm{e}^{\delta}H(y_1+1)  
    g(|y|)(n-1). 
\end{equation}
Combining~\eqref{eq:R1.stable}, \eqref{eq:R2.stable} and~\eqref{eq:R3.stable}
we obtain that 
\[
b_n(y)\le 
((C_g+1)C_H+
\gamma_{n-1}(\mathrm{e}^{\delta}\rho_{n-1}
+\delta \mathrm{e}^{\delta})
)g(|y|) H(y_1+1).
\]
Then, for 
\[
\gamma_n:=\max(\gamma_{n-1},(C_g+1)C_H+
\gamma_{n-1}(\mathrm{e}^{\delta}\rho_{n-1}
+\delta \mathrm{e}^{\delta})
)    
\]
the inequality~\eqref{eq:ind.gamma} holds.
Then, using the definition~\eqref{eq:delta.stable} of \(\delta \) 
it is not difficult 
to show that 
\[
\limsup_{n\to \infty}    \gamma_n<\infty.
\]
Hence, the statement of the lemma holds with
\[
\gamma:=\sup_n    \gamma_n<\infty. 
\]
\end{proof}



\section{Probabilities of large deviations when random walks start at 
an arbitrary starting point}

\begin{proof}[Proof of Theorem~\ref{thm.large.deviations}]  
Let 
\begin{align*} 
A(x,y)&=\left\{z: |y-z|\ge \frac{1}{2}|y-x|, z_1 \in  (0,x_1\wedge (y_1+1) ]  \right\}
\\
C(x,y)&=\left\{z: |x-z|\ge \frac{1}{2}|y-x|, z_1 \in  (0,x_1\wedge (y_1+1) ]  \right\}.
\end{align*}
If $|y-z|\ge \frac{1}{2}|y-x| $ then, by Lemma~\ref{lem:heavy.main.stable}, 
\[
b_{n-k}(y-z)\le \gamma_0 H(y_1-z_1+1)
g(|y-z|)
\le \gamma_0 H(y_1-z_1+1) g\left(\frac{|x-y|}{2}\right).
\]
This implies that
\begin{align*}
    &\sum_{k=0}^n 
    \int_{A(x,y)}
    p_k^+(0,dz-x)b_{n-k}(y-z)\\
    &\hspace{1cm}\le 
    \gamma_0 H(y_1+1) g\left(\frac{|x-y|}{2}\right) 
    \sum_{k=0}^n 
    \pr(x+S(k) \in A(x,y),\tau^+>k)
    \\
    &\hspace{1cm}\le 
    \gamma_0 H(y_1+1) g\left(\frac{|x-y|}{2}\right) 
    \sum_{k=0}^n 
    \pr(x_1+S_1(k) \in (0,x_1\wedge (y_1+1) ],\tau^+>k)\\ 
    &\hspace{1cm}\le \gamma_0 H(y_1+1) g\left(\frac{|x-y|}{2}\right) V(x_1). 
\end{align*}   
By the same argument,
\begin{align*}
    &\sum_{k=0}^n 
    \int_{C(x,y)}
    p_k^+(0,dz-x)b_{n-k}(y-z)\\
    &\hspace{0.5cm}\le
    \sum_{k=0}^n 
    \sum_{v\in \mathbb Z^d: 
    (v+\Delta)\cap C(x,y)\neq \emptyset
    }
    \pr(S(k)\in v-x+\Delta,\tau^+>k) 
    \max_{z\in v+\Delta}b_{n-k}(y-z)\\
    &\hspace{0.5cm}\le
    \sum_{k=0}^n 
    \sum_{v\in \mathbb Z^d: 
    (v+\Delta)\cap C(x,y)\neq \emptyset
    }
    \pr(S(k)\in v-x+\Delta,\tau^+>k) 
    \sum_{e_j\in \{0,1\}^d} 
    b_{n-k}(y-v-e_j)\\
    &\hspace{0.5cm}\le c 2^d 
    V(x_1+1)g\left(\frac{|x-y|}{2}-1\right)
    H(y_1+1). 
\end{align*}
These estimates give the desired bound.
\end{proof}    


\section{Asymptotics for the Green function near the boundary}

\begin{proof}[Proof of Theorem~\ref{thm.green.near.boundary}]
We consider the lattice case only.
    Fix $A>0$.  
    Then 
    \[
    G(x,y) =
    G_1(x,y)+G_2,(x,y):= 
    \sum_{n:Ac_n< |x-y|} p_n(x,y) 
    +\sum_{n:Ac_n\ge |x-y|} p_n(x,y). 
    \]
    Using Theorem~\ref{thm.large.deviations} 
    we obtain 
    \begin{align*} 
        G_1(x,y)&\le C 
        H(y_1)V(x_1) 
        \sum_{n:Ac_n< |x-y|} g(|x-y|) \\
        &\le 
        C H(y_1)V(x_1) 
        \frac{g(|x-y|)}{\mu(|x-y|/A)} \\ 
        &\le \frac{C}{A^{\alpha}}
        \frac{H(y_1)V(x_1)}{|x-y|^d}.
    \end{align*}    

    For the second term make use of Theorem~\ref{thm:small.deviations.arbitrary.lattice} 
    \begin{align*}
        G_2(x,y) 
        \sim  H(y_1)V(x_1)  
        \sum_{n:Ac_n\ge  |x-y|}  
        \frac{g_{\alpha,\sigma}(0, \frac{y_{2,d}-x_{2,d}}{c_n}) }{nc_n^d}
    \end{align*}    
    We will now analyse the series. 
    Using the regular variation of 
    $c_n$ we can write it as 
    \begin{align*}
        \sum_{n:Ac_n\ge  |x-y|}  
        \frac{g_{\alpha,\sigma}(0, \frac{y_{2,d}-x_{2,d}}{c_n}) }{nc_n^d}    
        &\sim 
        \sum_{n:n \ge  A^{-\alpha}/\mu(|x-y|)}  
        \frac{g_{\alpha,\sigma}(0, \frac{y_{2,d}-x_{2,d}}{c_n}) }{nc_n^d}\\ 
        &\sim
        \int_{A^{-\alpha}/\mu(|x-y|)}^\infty 
        \frac{g_{\alpha,\sigma}(0, \frac{y_{2,d}-x_{2,d}}{c_{[t]}}) }{tc_{[t]}^d}dt \\ 
        &\sim \int_0^{\mu(|x-y|)A^{\alpha}} 
        \frac{g_{\alpha,\sigma}(0, \frac{y_{2,d}-x_{2,d}}{c_{[1/z]}}) }{zc_{[1/z]}^d}dz\\ 
        &\sim 
        \frac{1}{|x-y|^d} 
        \int_0^{A^{\alpha}} 
        g_{\alpha,\sigma}\left(0, \frac{y_{2,d}-x_{2,d}}{|x-y|}z^{1/\alpha}\right) z^{d/\alpha-1}dz, 
    \end{align*}
    as $|x-y|\to \infty$.
 Here we used the fact that 
 since $c_n$ is regularly varying  
 of index $1/\alpha$, 
 for a fixed $z>0$, 
 \[
   c([1/(\mu(|x-y|)z)]) 
   \sim z^{-1/\alpha} 
   c([1/(\mu(|x-y|))]) 
   \sim z^{-1/\alpha} |x-y|,
 \]   
 as $|x-y|\to \infty$. 
Thus, in the general case,  
\[ 
G_2(x,y) 
\sim  C  \frac{H(y_1)V(x_1)}{|x-y|^d}  
\int_0^{A^{\alpha}} 
        g_{\alpha,\sigma}\left(0, \frac{y_{2,d}-x_{2,d}}{|x-y|}z^{1/\alpha}\right) z^{d/\alpha-1}dz. 
\]
Letting $A\to\infty$ and substituting $z^{1/\alpha}=t$ we arrive at the conclusion. 
Noting that 
    in the isotropic case the ration  
    $\frac{y_{2,d}-x_{2,d}}{\left| x-y \right|}$
    belongs asymptotically to the unit sphere,  
    we obtain the result in this case as well.

To obtain the upper bound~\eqref{eq.g.upper.bound} in the analysis of the second term 
we make use of 
Lemma~\ref{lem:lem20} instead of of Theorem~\ref{thm:small.deviations.arbitrary.lattice}.

\end{proof}


\begin{thebibliography}{99}


    
    
   

\bibitem{Berger2019} Berger, Q. 
\newblock Strong renewal theorems and local large deviations for multivariate random walks and renewals. 
\newblock {\em Electron. J. Probab.} \textbf{24}: paper 46, 47pp, 2019.


\bibitem{BGT} Bingham N.H., Goldie C.M., Teugels J.L.
\newblock {\em Regular variation.} 
\newblock Cambridge: Cambridge University Press, 1987, 494 pp.


\bibitem{CaravennaChaumont2013} Caravenna, F. and Chaumont, L. 
\newblock An invariance principle for random walk bridges conditioned to stay positive. 
\newblock {\em Electron. J. Probab.} \textbf{18}: paper 60, 32 pp.,  2013.





    
    

   
    
    
    
    \bibitem{DW10} Denisov, D. and V. Wachtel, V.
    \newblock Conditional limit theorems for ordered random walks.
    \newblock \emph{Elec. J. Probab.}, \textbf{15}: 292-322, 2010.
    
    \bibitem{DW15} Denisov, D. and Wachtel, V.
    \newblock Random walks in cones.
    \newblock \emph{Ann. Probab.}, \textbf{43}: 992-1044, 2015.
    
    
   \bibitem{DW16} Denisov, D. and Wachtel, V.
   \newblock Exact asymptotics for the moment of crossing a curved boundary by an asymptotically stable random walk
   \newblock {\em Theory of Prob. Appl.} \textbf{60}(3): 481-500, 2016.
    
    \bibitem{DW19} Denisov, D. and Wachtel, V.
    \newblock Alternative constructions of a harmonic function for a random walk in a cone.
    \newblock {\em Elec. J. Probab.}, \textbf{24}: paper no. 92, 26pp, 2019.

    \bibitem{DW21} Denisov, D. and Wachtel, V.
    \newblock Random walks in cones revisited. 
    \newblock {\em arXiv:}2112.10244, 2021. 

    \bibitem{Doney1985} Doney, R. A. 
    \newblock   Conditional limit theorems for asymptotically stable random walks.  
    \newblock {\em Z. Wahrsch. Verw. Gebiete} \textbf{70}: 351–360, 1985.


    \bibitem{Doney2012} Doney, R. A. 
    \newblock   Local behaviour of first passage probabilities. 
    \newblock {\em Probab. Theory Relat. Fields} \textbf{152}: 559–88, 2012.

    \bibitem{DJ2012} Doney, R. and Jones, E. 
    \newblock Large deviation results for random walks conditioned to stay positive. 
    \newblock \emph{Electron. Commun. Probab.} \textbf{17}: paper no. 38, 1-11, 2012.
 

    \bibitem{DuW15} Duraj, J. and Wachtel, V.
    \newblock Invariance principles for random walks in cones.
    \newblock  \emph{Stochastic Process. Appl.} \textbf{130}, 3920-3942, 2020. 

    \bibitem{DRTW18} Duraj J., Raschel K., Tarrago, P. and Wachtel V. 
    \newblock Martin boundary of random walks in convex cones
    \newblock \emph{ Ann. H. Lebesgue} \textbf{5},559-609, 2022.
    
    \bibitem{FE} Feller W.
    \newblock{\em An Introduction to Probability Theory and its Applications.} 
    \newblock V.2, Willey, New York-London-Sydney-Toronto, 1971.


   


    
    
    
    \bibitem{Nagaev79} Nagaev, S. V. 
    \newblock Large Deviations of Sums of Independent Random Variables. 
    \newblock \emph {Ann. Probab.} \textbf{7}: 745-–789, 1979.


    
    
    \bibitem{Rvacheva62} Rvacheva, E.L.
    \newblock  On domains of attraction of multi-dimensional distributions. 
    \newblock Select. Transl. Math. Statist. and Probability, Vol. 2, American Mathematical Society, Providence, RI , pp. 183-205, 1962.     

    \bibitem{Rog71} Rogozin B.A. 
    \newblock On the distribution of the first ladder moment and height and fluctuations of a random walk.
    \newblock{\em Theory Probab. Appl.,} \textbf{16}: 575-595, 1971.


    



    \bibitem{TT2002} Tamura, Y. and Tanaka, H.  
    \newblock On a Fluctuation Identity for Multidimensional L\'evy Processes. 
    \newblock \emph{Tokyo J. Math.} \textbf{25}: 363–380, 2002.

    \bibitem{TT2008} Tamura, Y. and Tanaka, H.  
    \newblock On a formula on the potential operators of absorbing L\'evy processes in the half space. 
    \newblock \emph{Stoch. Proc. Appl.} \textbf{118}(2): 199-212, 2008.


    \bibitem{U2014} Uchiyama, K. 
    \newblock Green’s functions of random walks on the upper half plane. 
    \newblock \emph{Tohoku Math. J.} \textbf{66}(2): 289–307, 2014.


   


   \bibitem{VW2009} Vatutin, V. A. and Wachtel, V. 
   \newblock Local probabilities for random walks conditioned to stay positive. 
   \newblock {\em Probab. Theory Relat. Fields} \textbf{143}: 177–217, 2009.

   \bibitem{Williamson68} Williamson, J. A. 
    \newblock Random walks and Riesz kernels.  
   \newblock {\em Pacific J. Math.} \textbf{25}(2): 393-415, 1968.
    
   

\bibitem{Zol57} Zolotarev V.M. 
\newblock Mellin-Stiltjes transform in probability theory. 
\newblock{\em Theory Probab. Appl.,} 2: 433-460, 1957.
\
   
   \end{thebibliography}
    \end{document}